\documentclass[pdflatex,sn-mathphys-num]{sn-jnl}

\usepackage{graphicx}%
\usepackage{multirow}%
\usepackage{amsmath,amssymb,amsfonts}%
\usepackage{amsthm}%
\usepackage{mathrsfs}%
\usepackage[title]{appendix}%
\usepackage{xcolor}%
\usepackage{textcomp}%
\usepackage{manyfoot}%
\usepackage{booktabs}%
\usepackage{algorithm}%
\usepackage{algorithmicx}%
\usepackage{algpseudocode}%
\usepackage{caption}
\usepackage{bm}
\usepackage{listings}%

\theoremstyle{thmstyleone}%
\newtheorem{theorem}{Theorem}
\newtheorem{lemma}{Lemma}%
\newtheorem{coro}{Corollary}%
\newtheorem{proposition}[theorem]{Proposition}%

\theoremstyle{thmstyletwo}%
\newtheorem{remark}{Remark}%

\theoremstyle{thmstylethree}%
\newtheorem{definition}{Definition}%

\newcommand{\mat}[1]{\ensuremath{\mathsf{#1}}}
\newcommand{\fnc}[1]{\ensuremath{\mathcal{#1}}}

\DeclareMathOperator{\diag}{diag}

\begin{document}

\title[Modal Collocation]{A collocation scheme that is equivalent to discontinuous Galerkin discretizations}


\author*[1]{\fnm{Jason E.} \sur{Hicken}}\email{hickej2@rpi.edu}

\affil*[1]{\orgdiv{Department of Mechanical, Aerospace, and Nuclear Engineering}, \orgname{Rensselaer Polytechnic Institute}, \orgaddress{\street{110 8th Street}, \city{Troy}, \postcode{12180}, \state{New York}, \country{United States of America}}}


\abstract{
A spectral collocation operator with the summation-by-parts property was introduced by Chan to develop entropy-stable discontinuous Galerkin (DG) semi-discretizations (\href{https://doi.org/10.1016/j.jcp.2018.02.033}{DOI:j.jcp.2018.02.033}).  The present work shows that semi-discretizations based on this collocation operator produce solutions that are equivalent to solutions of a DG semi-discretization using the same underlying quadrature.  The equivalence holds regardless of the number of degrees of freedom in the collocation scheme and when the quadrature is not strictly positive.  Extraneous degrees of freedom in the collocation scheme are associated with the nullspace of the operator and remain zero throughout an unsteady simulation.  If necessary, nullspace consistency can be recovered by introducing projection-based numerical dissipation that targets only the extraneous modes.  The equivalence between collocation and DG solutions is verified for the constant-coefficient advection equation and Burgers' equation on triangular meshes. The numerical results show that equivalence breaks down for entropy-stable semi-discretizations of Burgers' equation based on a skew-symmetric splitting, but that equivalence can be recovered by projecting the collocation scheme's residual onto the relevant polynomial space.  In addition to investigating equivalence, the results demonstrate that the collocation operator produces semi-discretizations with favorable spectral radii compared with a commonly used summation-by-parts operator construction.}

\keywords{discontinuous Galerkin, summation-by-parts, spectral element, stability}


\pacs[MSC Classification]{65M06, 65M70, 65M12}

\maketitle

\section{Introduction}\label{sec:intro}

By mimicking integration-by-parts at the discrete level, summation-by-parts (SBP) operators~\cite{Kreiss1974finite} provide a versatile foundation for constructing stable, high-order discretizations.  This foundation has been exploited for many years, but interest in SBP operators grew significantly after 2013, when Fisher and Carpenter \cite{Fisher2013highorder} demonstrated that SBP operators could be used to construct high-order, entropy-stable schemes on bounded domains.

SBP operators were originally developed as finite-difference discretizations~\cite{Kreiss1974finite}, but they have since been generalized to unstructured, element-based grids.  This includes both tensor-product operators~\cite{Gassner2013skew,Carpenter2014entropy} and multidimensional SBP operators~\cite{Hicken2016multidimensional}.  Enabling SBP discretizations for unstructured grids has been a key development in allowing the framework to tackle more complex problems.



Multidimensional SBP operators can be difficult to use, and the goal that originally motivated this work was to identify a straightforward SBP operator with a closed-form expression.  The search for such an operator brought me to the projection-based SBP operator introduced by Chan~\cite{Chan2018discretely} --- see also \cite{Chan2019skew}.  I will refer to Chan's operator as a modal collocation (MC) operator for reasons that will become clear shortly.  Chan did not use the MC operator in a collocation scheme, but as an intermediate step in a finite-element discretization.  To the best of my knowledge, only Montoya and Zingg~\cite{Montoya2024efficient} have used the MC operator in a collocation setting to compare their tensor-product simplex operators with multidimensional SBP operators.

The contribution of this work is to highlight a remarkable property of modal collocation: it produces a solution that is identical to a discontinuous Galerkin (DG) discretization based on the same quadrature.  Thus, regardless of the number of quadrature points, a modal-collocation discretization is effectively a modal DG method.  This is despite the fact --- shown here --- that MC operators are nullspace inconsistent~\cite{Svard2019convergence}, which is typically detrimental for finite difference operators; see, for example, \cite{Glaubitz2026SBPnotenough}.

Here is a brief outline of the remaining paper.  In Section~\ref{sec:methods}, I present the MC operator definition and review some results from~\cite{Chan2018discretely}.  The main theoretical contributions related to MC-DG equivalence are presented in Section~\ref{sec:analysis}.  I verify the theory in Section~\ref{sec:results} by demonstrating the equivalence between a few MC and DG discretizations.  Finally, I provide a summary and discussion of the paper's results in Section~\ref{sec:conclude}

\section{Definition and review of modal-collocation discretizations}\label{sec:methods}

\subsection{Notation and preliminaries}\label{sec:prelim}


Vectors are denoted with bold font, such as the spatial coordinates $\bm{x} = [x_1,x_2,x_3]^T$.  Matrices use sans-serif font, e.g., $\mat{A} \in \mathbb{R}^{N \times M}$ is an $N \times M$ matrix.  Element domains will be represented as $\Omega \subset \mathbb{R}^{D}$, where $D \geq 1$ is the space dimension, and the notation $\Gamma = \partial \Omega$ indicates the boundary of $\Omega$. I will use calligraphic type in order to distinguish functions from point values.  For example, $\fnc{U}(x,t) \in L^2(\Omega)$ is often used to represent the solution to a partial-differential equation, whereas $\bm{u} \in \mathbb{R}^N$ is used to represent its discrete or semi-discrete counterpart.

I will use polynomials to specify the accuracy of various operators, so some notation is needed to describe these functions.  Let $\mathbb{P}^{P}(\Omega)$ denote the space of polynomials of total degree $P$ on the domain $\Omega$.  The dimension of this space is denoted $N_P = \binom{P+D}{D}$. The ordered set $\{ \fnc{V}_i \}_{i=1}^{N_P}$ represents an orthonormal basis for $\mathbb{P}^{P}(\Omega)$; that is, the basis functions are orthonormal with respect to the $L^2$ inner product on $\Omega$:
\begin{equation}\label{eq:ortho}
    \int_{\Omega} \fnc{V}_i \fnc{V}_j \,d\Omega = \delta_{ij},
    \qquad\forall\, i,j = 1, 2, \ldots, N_P,
\end{equation}
where $\delta_{ij}$ is the Kronecker delta.

I will also need some notation for quadrature rules, since numerical integration plays a key role in the definition and analysis of MC operators.  The nodes and weights for an $N$-point quadrature rule over the domain $\Omega$ will be represented as $X_{\Omega} = \{\bm{x}_i\}_{i=1}^{N}$ and $W_{\Omega} = \{ w_{i} \}_{i=1}^{N}$, respectively.  Similarly, I will use $X_{\Gamma} = \{\hat{\bm{x}}_i\}_{i=1}^{M}$ and $W_{\Gamma} = \{ \hat{w}_{i} \}_{i=1}^{M}$ to denote the nodes and weights, respectively, for an $M$-point quadrature rule over $\Gamma$.  Unless stated otherwise, I assume that the quadrature rules are positive-interior rules; for example, this implies that $X_{\Omega} \subset \Omega$, and $w > 0$ for all $w \in W_{\Omega}$.

Some matrices will arise frequently and warrant being defined here.  The matrix $\mat{I}_{N} \in \mathbb{R}^{N\times N}$ denotes the $N\times N$ identity matrix, and the matrix $\mat{0}_{N\times M} \in \mathbb{R}^{N \times M}$ is the $N \times M$ matrix of zeros.  The matrix $\mat{V} \in \mathbb{R}^{N \times N_P}$ is reserved to represent the Vandermonde matrix consisting of the orthonormal basis $\{ \fnc{V}_i \}_{i=1}^{N_P}$ evaluated at the nodes $X_{\Omega}$.  Similarly, the matrix $\mat{V}_{x_d} \in \mathbb{R}^{N \times N_P}$ contains the derivatives of the orthonormal basis, with respect to $x_d$, evaluated at the nodes $X_{\Omega}$.  To be precise, the entries in these matrices are 
\begin{equation}\label{eq:vandermonde}
    [\mat{V}]_{ij} = \fnc{V}_{j}(\bm{x}_i), \qquad 
    \text{and}\qquad 
    [\mat{V}_{x_d}]_{ij} = \frac{\partial \fnc{V}_{j}}{\partial x_d}(\bm{x}_i),
\end{equation}
where $i=1,2,\ldots,N$, and $j=1,2,\ldots,N_P$.

Finally, I review the definition of diagonal-norm, first-derivative multidimensional SBP operators~\cite{Hicken2016multidimensional}, since this definition will be referenced several times in the paper.

\begin{definition}[Diagonal-norm, first-derivative summation-by-parts operator]\label{def:sbp}
The matrix $D_{x_d} \in \mathbb{R}^{N\times N}$ is a degree $P$, diagonal-norm, summation-by-parts operator approximating the derivative operator $\partial/\partial x_d$, $d=1,2,\ldots,D$, at the nodes $X_{\Omega} = \{x_{i}\}_{i=1}^{N}$ if it satisfies the following conditions.
\begin{description}
    \item[Derivative Accuracy:] The matrix differentiates polynomials in $\mathbb{P}^P(\Omega)$ exactly.  Thus,
    \begin{equation}\label{eq:sbp_accuracy}
        \mat{D}_{x_d} \mat{V} = \mat{V}_{x_d}
    \end{equation}
    where $\mat{V}$ and $\mat{V}_{x_d}$ are the matrices defined in~\eqref{eq:vandermonde}.
    \item[Factorization:] The matrix can be factored as $\mat{D}_{x_d} = \mat{W}^{-1} \mat{Q}_{x_d}$, where $\mat{W}$ is a diagonal, positive-definite matrix.
    \item[Boundary Accuracy:] The symmetric matrix $\mat{E}_{x_d} = \mat{Q}_{x_d} + \mat{Q}_{x_d}^{T}$ satisfies
    \begin{equation}\label{eq:sbp_boundary}
    \big[ \mat{V}^T \mat{E}_{x_d} \mat{V} \big]_{ij} = \int_{\Gamma} \fnc{V}_i \fnc{V}_j \, n_{x_d}\, d\Gamma,
    \end{equation}
    where $n_{x_d}$ is the $x_{d}$ component of the outward unit normal vector on $\Gamma$.
\end{description}
\end{definition}

While the SBP definition does not mention quadrature, one can show that the nodes $X_{\Omega}$ and the corresponding diagonal entries in $\mat{W}$ constitute a degree $2P-1$ exact quadrature rule over $\Omega$~\cite{Hicken2013quad,Hicken2016multidimensional}.  Thus, my choice of notation for the nodes, $X_{\Omega}$, in the SBP definition reflects this connection to quadrature.

\subsection{First-derivative modal-collocation operator}

I begin this section by introducing the closed-form expression for the first-derivative operator used in modal-collocation discretizations.  The definition below uses orthonormal basis functions, but it is otherwise identical to the projection-based derivative operator in Equation (24) of \cite{Chan2018discretely}.  I adopt an orthonormal basis to simplify the presentation and analysis, but it is not necessary in practice; see Section~\ref{sec:generalizations}.

\begin{definition}[First-derivative modal-collocation operator]\label{def:mc}
    Let $\{ X_{\Omega}, W_{\Omega}\}$ be a positive-interior quadrature rule that is exact for polynomials of total degree $2P$.  Then the matrix $\mat{D}_{x_d} \in \mathbb{R}^{N\times N}$ is a degree $P$ modal-collocation operator corresponding to the first derivative $\partial/\partial x_{d}$,  $d=1,2,\ldots,D$, at the quadrature nodes $X_{\Omega}$ if it has the form 
    \begin{equation*}
        \mat{D}_{x_d} = \mat{V}_{x_d} \mat{V}^T \mat{W},
    \end{equation*}
    where $\mat{W} = \diag(w_1, w_2, \ldots,w_N)$ is a positive-definite, diagonal matrix whose nonzero entries are the quadrature weights in $W_{\Omega}$. 
\end{definition}

Chan~\cite[Lemma~1]{Chan2018discretely} demonstrated that MC operators are degree $P$ SBP operators.  I restate and prove this result below in the context of orthonormal bases.  While the result is not novel, I include it for completeness and to introduce some definitions needed later.


\begin{theorem}\label{thm:mc_are_sbp}
All degree $P$ first-derivative MC operators are also degree $P$, diagonal-norm SBP operators.
\end{theorem}

\begin{proof}
The derivative-accuracy condition follows by multiplying the MC operator and the Vandermonde matrix $\mat{V}$:
\begin{equation*}
    \mat{D}_{x_d} \mat{V} = \mat{V}_{x_d} \mat{V}^T \mat{W} \mat{V}
     = \mat{V}_{x_d},
\end{equation*}
where I used the identity $\mat{V}^T \mat{W} \mat{V} = \mat{I}_{N_P}$, which follows from the orthonormality of the basis, c.f.~\eqref{eq:ortho}, and the $2P$ exactness of the quadrature used in the MC operator definition.

The factorization condition in Definition~\ref{def:sbp} is trivially satisfied by defining 
\begin{equation}\label{eq:Q_identity}
    \mat{Q}_{x_d} \equiv \mat{W} \mat{D}_{x_d} = \mat{W} \mat{V}_{x_d} \mat{V}^T \mat{W}.
\end{equation}
This definition of $\mat{Q}_{x_d}$ implies that the symmetric matrix 
\begin{equation}\label{eq:E_identity}
    \mat{E}_{x_d} = \mat{Q}_{x_d} + \mat{Q}_{x_d}^T
    = \mat{W} \mat{V}_{x_d} \mat{V}^T \mat{W} + \mat{W} \mat{V} \mat{V}_{x_d}^T \mat{W},
\end{equation}
which can be used to verify the boundary-accuracy condition, Equation~\eqref{eq:sbp_boundary}, as follows:
\begin{alignat*}{2}
\big[ \mat{V}^T \mat{E}_{x_d} \mat{V} \big]_{ij}
&= \Big[ \mat{V}^T \Big( \mat{W} \mat{V}_{x_d} \mat{V}^T \mat{W} + \mat{W} \mat{V} \mat{V}_{x_d}^T \mat{W} \Big) \mat{V} \Big]_{ij},\qquad& &\text{(using \eqref{eq:E_identity})} \\
&= \big[ \mat{V}^T \mat{W} \mat{V}_{x_d} + \mat{V}_{x_d}^T \mat{W} \mat{V} \big]_{ij}, & &\text{(using $\mat{V}^T \mat{W} \mat{V} = \mat{I}_{N_P}$)}\\
&= \sum_{k=1}^{N} \fnc{V}_i(\bm{x}_k) w_k \frac{\partial \fnc{V}_j}{\partial x_d}(\bm{x}_k) + \sum_{k=1}^{N} \frac{\partial \fnc{V}_i}{\partial x_d} (\bm{x}_k) w_k \fnc{V}_j (\bm{x}_k) && \\
&= \int_{\Omega} \Big( \fnc{V}_i \frac{\partial \fnc{V}_j}{\partial x_d} + \frac{\partial \fnc{V}_i}{\partial x_d} \fnc{V}_{j} \Big) \, d\Omega &&\text{(using quad. accuracy)}\\
&= \int_{\Omega} \fnc{V}_i \fnc{V}_j \, n_{x_d}\, d\Gamma, &&\text{(using div. theorem)}
\end{alignat*}
as required.  Thus, all three conditions are satisfied, which concludes the proof.
\end{proof}


The converse of Theorem~\ref{thm:mc_are_sbp} does not hold.  Most first-derivative SBP operators are not first-derivative MC operators, and I will provide an example in Section~\ref{sec:results}.  However, it is possible to convert an existing SBP operator into an MC operator by projecting it onto the modal basis.  Indeed, as I show below, this conversion is possible for any first-derivative operator that acts at the nodes of a sufficiently high-order quadrature rule.

\begin{proposition}\label{prop:convert}
Let $\mat{D}_{x_d} \in \mathbb{R}^{N \times N}$ be a nodal operator that exactly differentiates degree $P$ polynomials evaluated at the nodes $X_{\Omega}$; thus, $\mat{D}_{x_d} \mat{V} = \mat{V}_{x_d}$.  If the nodes coincide with a degree $2P$ exact quadrature rule, with positive weights $W_{\Omega}$, then $\mat{V} \mat{V}^T \mat{W} \mat{D}_{x_d} \mat{V} \mat{V}^T \mat{W}$ is a degree $P$ modal collocation operator.
\end{proposition}

\begin{proof}
The result follows immediately from the accuracy of the nodal operator $\mat{D}_{x_d}$, since
\begin{equation*}
    \mat{V}\, \mat{V}^T\, \mat{W}\, \mat{D}_{x_d}\, \mat{V}\, \mat{V}^T \,\mat{W}
    = \mat{V}\, \mat{V}^T\, \mat{W}\, \mat{V}_{x_d}\, \mat{V}^T\, \mat{W}
    = \mat{V}_{x_d} \, \mat{V}^T\, \mat{W},
\end{equation*}
where I used $\mat{V}\, \mat{V}^T\, \mat{W}\, \mat{V}_{x_d} = \mat{V}_{x_d}$ in the second step.
\end{proof}

\begin{remark}
Proposition~\ref{prop:convert} is an interesting observation, but it is not useful in practice: if one can perform the necessary projection $\mat{V}\mat{V}^{T} \mat{W}$, why not use the MC definition \ref{def:mc} directly?
\end{remark}

\subsection{Modal-collocation boundary operator}

There is an equivalent representation of the boundary operator $\mat{E}_{x_d}$ that can be used in practice to evaluate boundary integrals in the weak form.  This is the subject of Lemma~1 from \cite{Chan2018discretely}, which is restated below using the present notation.  However, before restating the lemma, I need the following definition that relates volume and boundary quadrature rules.  Equation~\eqref{eq:compatibility}, appearing in Definition~\ref{def:compatibility}, is called the compatibility condition in the SBP literature~\cite{Fernandez2014generalized}, and it is related to Assumption 1 in \cite{Chan2018discretely}.

\begin{definition}[Compatible volume and boundary quadratures]\label{def:compatibility}
    Let $\{X_\Omega, W_{\Omega}\}$ be a quadrature rule over $\Omega$ and let $\{X_\Gamma, W_{\Gamma}\}$ be a quadrature rule over its boundary $\Gamma$.  We say that the quadrature rules are degree-$P$ compatible if 
    \begin{equation}\label{eq:compatibility}
        \mat{V}^T \mat{W} \mat{V}_{x_d} + \mat{V}_{x_d}^T \mat{W} \mat{V}
        = \mat{V}_{\Gamma}^T\, \mat{W}_{\Gamma}\, \mat{N}_{x_d}\, \mat{V}_{\Gamma},
    \end{equation}
    where $\mat{W}$, $\mat{V}$, and $\mat{V}_{x_d}$ are defined as before, $\mat{W}_{\Gamma}$ is a diagonal matrix whose nonzero entries are the quadrature weights $W_{\Gamma}$, and $\mat{N}_{x_d}$ is a diagonal matrix holding the unit normal $n_{x_d}$ evaluated at the quadrature nodes $X_{\Gamma}$.
\end{definition}

Definition~\ref{def:compatibility} states that compatible volume and boundary quadrature rules must mimic the integration-by-parts formula for degree $P$ polynomials.  Written out in detail, Equation~\eqref{eq:compatibility} reads
\begin{multline*}
    \sum_{k=1}^{N} \fnc{V}_i(\bm{x}_k) w_k \frac{\partial \fnc{V}_j}{\partial x_d}(\bm{x}_k) + \sum_{k=1}^{N} \frac{\partial \fnc{V}_i}{\partial x_d} (\bm{x}_k) w_k \fnc{V}_j (\bm{x}_k) \\
    = \sum_{m=1}^{M} \fnc{V}_i(\hat{\bm{x}}_m) \fnc{V}_j(\hat{\bm{x}}_m) n_{x_d}(\hat{\bm{x}}_m) \hat{w}_m,\qquad \forall i,j = 1,2,\ldots,N_P.
\end{multline*}
The above identity is easy to enforce when the boundary $\Gamma$ consists of piecewise-linear facets, since $n_{x_d}$ is then constant over each sub-boundary.  In this case, one can choose a degree $2P$ exact boundary quadrature for each sub-boundary and a degree $2P-1$ exact volume quadrature.  Therefore, compatibility is straightforward to attain for standard reference elements, such as quadrilaterals, hexahedrons, triangles, and tetrahedrons.  Elements with curvilinear boundaries may require a higher than $2P$-exact boundary quadrature to achieve compatibility. On the other hand, quadratures that are only compatible for constant functions ($P=0$) are often adequate in practice for stable, high-order discretizations; see, for example,~\cite[Sec.~5]{Crean2018entropy}.

I can now restate Lemma 1 from \cite{Chan2018discretely}.


\begin{lemma}[\cite{Chan2018discretely}]\label{thm:mc_boundary}
    Let $\mat{D}_{x_d} = \mat{V}_{x_d} \mat{V}^T \mat{W} = \mat{W}^{-1} \mat{Q}_{x_d}$ be a degree $P$ first-derivative MC operator.  In addition, let $\{ X_{\Gamma}, W_{\Gamma} \}$ be a positive-interior quadrature rule over the boundary $\Gamma$ that is degree-$P$ compatible with the volume quadrature used to define $\mat{D}_{x_d}$.  Then 
    \begin{equation}\label{eq:mc_boundary}
        \mat{E}_{x_d} = \mat{R}_{\Gamma}^{T}\, \mat{W}_{\Gamma}\, \mat{N}_{x_d}\, \mat{R}_{\Gamma},
    \end{equation}
    where $\mat{R}_{\Gamma} = \mat{V}_{\Gamma}\, \mat{V}^T\, \mat{W}$, and $\mat{V}_{\Gamma}$ is the orthonormal basis evaluated at the boundary quadrature nodes $X_{\Gamma}$.
\end{lemma}

\section{Analysis of modal-collocation discretizations}\label{sec:analysis}

\subsection{Equivalence of modal-collocation and modal DG}

I will demonstrate the equivalence between MC and modal DG for a scalar, nonlinear hyperbolic equation with a source. Thus, I consider the initial boundary-value problem (IBVP)
\begin{equation}\label{eq:hyperbolic}
\begin{alignedat}{2}
    &\frac{\partial \fnc{U}}{\partial t} = -\sum_{d=1}^{D} \frac{\partial \fnc{F}_d (\fnc{U})}{\partial x_d} + \fnc{S}(\fnc{U},t),&\qquad &\forall\, \bm{x} \in \Omega, t \in [0,T], \\
    &\sum_{d=1}^{D} \fnc{F}_d(\fnc{U}) n_{x_d} = \fnc{F}_n^*(\fnc{U},\fnc{G}), &\qquad&\forall\, \bm{x} \in \Gamma, t \in [0,T], \\
    &\fnc{U}(\bm{x},0) = \fnc{U}_0(\bm{x}),& \qquad &\forall\, \bm{x} \in \Omega.
\end{alignedat}
\end{equation}
The source term, $\fnc{S}(\fnc{U},t) \in L^2(\Omega)$, is nonlinear and time dependent, in general.  Boundary conditions are specified using the flux function $\fnc{F}_n^*(\fnc{U},\fnc{G})$, which selects the boundary data $\fnc{G}(\bm{x},t)$ where the characteristics enter the domain and the solution $\fnc{U}(\bm{x},t)$ otherwise.  The initial condition is specified using $\fnc{U}_0(\bm{x}) \in L^2(\Omega)$.

\paragraph{Modal DG semi-discretization}

For simplicity, I will assume one element is used for the domain $\Omega$.  The modal, discontinuous Galerkin semi-discretization of \eqref{eq:hyperbolic} begins with the weak formulation
\begin{multline}\label{eq:dg_weak}
    \int_{\Omega} \fnc{V}_h^T \, \frac{\partial \fnc{U}_h}{\partial t} \, d\Omega 
    = \sum_{d=1}^{D} \int_{\Omega} \frac{\partial \fnc{V}_h}{\partial x_d} \, \fnc{F}_d(\fnc{U}_h) \, d\Omega \\ - \int_{\Gamma} \fnc{V}_h \, \fnc{F}_n^*(\fnc{U}_h, \fnc{G}) \, d \Gamma + \int_{\Omega} \fnc{V}_h \, \fnc{S} \, d\Omega,
    \qquad \forall\, t \in [0,T],
\end{multline}
where the test functions are elements of the orthonormal basis, $\fnc{V}_h \in \{\fnc{V}_i \}_{i=1}^{N_P}$, and the DG solution is 
\begin{equation*}
\fnc{U}_h(\bm{x},t) = \sum_{i=1}^{N_P} \fnc{V}_{i}(\bm{x}) \tilde{u}_{i}(t).
\end{equation*}
The initial values of the modal coefficients, $\{ \tilde{u}_i(0) \}_{i=1}^{N_P}$, are found through an $L^2$ projection of the initial condition $\fnc{U}_{0}(\bm{x})$ onto the modal basis.  Thus,
\begin{equation}\label{eq:dg_ic_integral}
    \int_{\Omega} \fnc{V}_h \, \fnc{U}_h(\bm{x},0) \, d\Omega = \int_{\Omega} \fnc{V}_h \, \fnc{U}_0(\bm{x}) \, d\Omega, \qquad \forall\, \fnc{V}_h \in \{ \fnc{V}_{i} \}_{i=1}^{N_P}.
\end{equation}

Next, I approximate the integrals in \eqref{eq:dg_weak} and \eqref{eq:dg_ic_integral} using positive-interior quadrature rules that are degree $2P$ exact over $\Omega$ and $\Gamma$.  Substituting the matrices $\mat{V}$, $\mat{V}_{x_d}$, $\mat{V}_\Gamma$, $\mat{W}$, and $\mat{W}_{\Gamma}$ defined previously, and using the identity $\mat{V}^T \mat{W} \mat{V} = \mat{I}_{N_P}$, the DG weak formulation becomes
\begin{equation}\label{eq:dg}
    \frac{d \tilde{\bm{u}}}{dt} = \sum_{d=1}^D \mat{V}_{x_d}^T\, \mat{W} \,\bm{f}_{d}(\mat{V} \tilde{\bm{u}}) - \mat{V}_{\Gamma}^T \,\mat{W}_\Gamma \,\bm{f}^*_n(\mat{V}_{\Gamma} \tilde{\bm{u}}, \bm{g}_{\Gamma}) + \mat{V}^T \,\mat{W} \,\bm{s}(\mat{V}\tilde{\bm{u}},t),\qquad\forall\, t \in [0,T],
\end{equation}
where I have introduced the vector of modal coefficients, $\tilde{\bm{u}} = [\tilde{u}_1, \tilde{u}_2, \ldots, \tilde{u}_{N_P}]^T$, as well as the vector-valued functions $\bm{f}_{d} : \mathbb{R}^{N} \rightarrow \mathbb{R}^{N}$, $\bm{f}_n^* : \mathbb{R}^{M} \times \mathbb{R}^{M} \rightarrow \mathbb{R}^{M}$, and $\bm{s} : \mathbb{R}^{N} \times \mathbb{R} \rightarrow \mathbb{R}$.  These functions are defined entrywise below:
\begin{alignat}{2} 
\big[ \bm{f}_{d}(\mat{V} \tilde{\bm{u}}) \big]_i &= \fnc{F}_{d}\big(\fnc{U}_h(\bm{x}_i,t) \big),& \qquad &\forall\, i = 1,2,\ldots,N, \label{eq:fxd_def}\\
\big[ \bm{f}^*_n(\mat{V}_\Gamma \tilde{\bm{u}},\bm{g}_\Gamma) \big]_i &=
\fnc{F}^*_n\big( \fnc{U}_h(\bm{x}_i,t), \fnc{G}(\bm{x}_i,t) \big), & \qquad &\forall\, i = 1,2,\ldots,M, \label{eq:fstar_def} \\
\big[ \bm{s}(\mat{V} \tilde{\bm{u}},t) \big]_i &= \fnc{S}(\fnc{U}_h(\bm{x}_i),t),& \qquad &\forall\, i = 1,2,\ldots,N, \label{eq:s_def}
\end{alignat}
where $\big[ \bm{g}_\Gamma \big]_i = \fnc{G}(\bm{x}_i,t)$ for all $i=1,2,\ldots,M$.  In addition, the initial condition for the modal coefficients can be expressed as
\begin{equation}\label{eq:dg_ic}
\tilde{\bm{u}}(0) = \mat{V}^T \,\mat{W} \,\bm{u}_0, 
\end{equation}
where $\big[ \bm{u}_0 \big]_i = \fnc{U}_0(\bm{x}_i)$, for all $i = 1,2,\ldots, N$.

\paragraph{Modal collocation semi-discretization}

Following the SBP literature, the strong-form discretization of the IBVP \eqref{eq:hyperbolic} based on the first-derivative MC operator is
\begin{multline}\label{eq:mc}
    \frac{d \bm{u}}{dt} = - \sum_{d=1}^{D} \mat{D}_{x_d} \, \bm{f}_{d}(\bm{u}) -\mat{W}^{-1}\, \mat{R}_{\Gamma}^T \, \mat{W}_{\Gamma} \, \big( \bm{f}_n^*(\mat{R}_\Gamma \bm{u}, \bm{g}_{\Gamma}) - \bm{f}_{n}(\bm{u}) \big) + \mat{V}\, \mat{V}^T\, \mat{W} \, \bm{s}(\bm{u},t), \\ \forall\, t \in [0,T],
\end{multline}
where 
\begin{equation}\label{eq:fn}
\bm{f}_n(\bm{u}) = \sum_{d=1}^D \mat{N}_{x_d}\,\mat{R}_\Gamma\, \bm{f}_{d}(\bm{u}).
\end{equation}
The functions $\bm{f}_{d}$, $\bm{f}_n^*$, and $\bm{s}$ appearing in the MC discretization~\eqref{eq:mc} are the same as the ones defined in Equations~\eqref{eq:fxd_def}--\eqref{eq:s_def} for the DG weak formulation; however, they are nominally evaluated using the MC solution rather than the DG solution.

Notice that the source term in Equation~\eqref{eq:mc} has been projected onto the modal basis.  Similarly, rather than taking $\bm{u}(0) = \bm{u}_0$ as the initial condition, I adopt the projected initial condition 
\begin{equation}\label{eq:mc_ic}
    \bm{u}(0) = \mat{V}\, \mat{V}^T \, \mat{W}\, \bm{u}_0.
\end{equation}

\begin{remark}
Projecting the source term and initial condition onto the polynomial space is not commonly done in SBP discretizations but is necessary to prove MC-DG equivalence.
\end{remark}

The discretization \eqref{eq:mc} can be rewritten in an equivalent ``weak'' form by left multiplying by $\mat{W}$ and using the properties of the MC operator, specifically Equations \eqref{eq:Q_identity}, \eqref{eq:E_identity}, and \eqref{eq:mc_boundary}.  After some algebra I obtain 
\begin{multline}\label{eq:mc_weak}
    \mat{W} \frac{d \bm{u}}{dt} = \sum_{d=1}^D \mat{W} \,\mat{V} \, \mat{V}_{x_d}^T \, \mat{W} \, \bm{f}_{d}(\bm{u}) - \mat{R}_\Gamma^T \, \mat{W}_\Gamma \, \bm{f}_n^*(\mat{R}_\Gamma \bm{u}, \bm{g}_\Gamma) + \mat{W} \, \mat{V} \, \mat{V}^T \, \mat{W} \bm{s}(\bm{u},t),\\
    \forall\, t\in[0,T].
\end{multline}
This form of the MC discretization will be useful in the proof of MC-DG equivalence. 

\paragraph{Main result}

\begin{theorem}[MC-DG equivalence]\label{thm:equivalence}
Assume that the initial-value problem corresponding to the modal DG semi-discretization \eqref{eq:dg_weak} and \eqref{eq:dg_ic} is well posed.  Then the solution to the MC semi-discretization \eqref{eq:mc} and \eqref{eq:mc_ic} is equivalent to the modal DG solution.  That is,
\begin{equation*}
    \bm{u}(t) = \mat{V}\, \tilde{\bm{u}}(t), \qquad \forall\, t \in [0,T].
\end{equation*}
\end{theorem}

\begin{proof}
Let $\mat{Z} \in \mathbb{R}^{N \times N_Z}$, where $N_Z = N -  N_P$, be a matrix in the nullspace of $\mat{V}^T \, \mat{W}$ and orthonormal with respect to $\mat{W}$; that is,
\begin{equation}\label{eq:Z_def}
    \mat{V}^T \, \mat{W} \, \mat{Z} = \mat{0}_{N_P \times N_Z},
    \qquad\text{and}\qquad
    \mat{Z}^T\, \mat{W} \, \mat{Z} = \mat{I}_{N_Z}.
\end{equation}
Based on the rank-nullity theorem, the MC solution can be decomposed using the modal and nullspace bases as follows:
\begin{equation*}
    \bm{u}(t) = \mat{V} \, \bm{u}_{V}(t) + \mat{Z}\, \bm{u}_Z(t),
\end{equation*}
where $\bm{u}_{V}(t) \in \mathbb{R}^{N_P}$ and $\bm{u}_Z(t) \in \mathbb{R}^{N_Z}$.  The proof proceeds by showing that $\bm{u}_Z(t) = \bm{0}$ and $\bm{u}_V(t) = \tilde{\bm{u}}(t)$.

To show that $\bm{u}_Z(t) = \bm{0}$, left multiply the weak form of the MC discretization, Equation \eqref{eq:mc_weak}, by $\mat{Z}^{T}$.  On the left-hand side this yields (see \eqref{eq:Z_def})
\begin{equation*}
\mat{Z}^{T}\, \mat{W} \frac{d}{dt} \big(\mat{V}\, \bm{u}_{V} + \mat{Z}\, \bm{u}_Z \big) = \frac{d \bm{u}_Z}{dt},
\end{equation*}
while on the right-hand side we obtain 
\begin{equation*}
\sum_{d=1}^D \mat{Z}^T \, \mat{W}\, \mat{V}\, \mat{V}_{x_d}^T\, \mat{W}\, \bm{f}_{d}(\bm{u}) - \mat{Z}^T\, \mat{R}_\Gamma^T\, \mat{W}_\Gamma\, \bm{f}_n^*(\mat{R}_\Gamma \bm{u}, \bm{g}_\Gamma) + \mat{Z}^T\, \mat{W}\, \mat{V}\, \mat{V}^T\, \mat{W}\, \bm{s}(\bm{u},t) = \bm{0},
\end{equation*}
since, again, $\mat{Z}$ is orthogonal to $\mat{V}$ with respect to $\mat{W}$; recall from Lemma~\ref{thm:mc_boundary} that $\mat{R}_\Gamma = \mat{V}_\Gamma \, \mat{V}^T \,\mat{W}$ so $\mat{Z}^T\,\mat{R}_\Gamma^T = \mat{0}_{N_Z \times N_M}$.  Furthermore, since $\bm{u}(0) = \mat{V}\, \bm{u}_V(0) + \mat{Z}\, \bm{u}_Z(0) = \mat{V} \, \mat{V}^T \,\mat{W} \,\bm{u}_0$ from \eqref{eq:mc_ic}, it follows that the vector $\bm{u}_{Z}$ is governed by the initial value problem 
\begin{equation*}
    \frac{d \bm{u}_Z}{dt} = \bm{0}, \qquad \forall\, t \in [0,T], 
    \qquad\text{with IC}\qquad \bm{u}_Z(0) = \bm{0}.
\end{equation*}
Thus, $\bm{u}_Z(t) = \bm{0}$ for all time $t \in [0,T]$.  Consequently, the MC solution must be of the form $\bm{u}(t) = \mat{V} \bm{u}_{V}(t)$.

Next, I left multiply Equation \eqref{eq:mc_weak} by $\mat{V}^{T}$ and substitute $\bm{u}(t) = \mat{V} \, \bm{u}_{V}(t)$:
\begin{equation}\label{eq:uv}
\frac{d \bm{u}_V}{dt} = \sum_{d=1}^D \mat{V}_{x_d}^T \, \mat{W} \, \bm{f}_{d}(\mat{V} \bm{u}_{V}) - \mat{V}_\Gamma^T\, \mat{W}_\Gamma \, \bm{f}_n^*(\mat{V}_\Gamma \bm{u}_{V}, \bm{g}) + \mat{V}^T \, \mat{W}\, \bm{s}(\mat{V} \bm{u}_{V}, t),\qquad \forall\, t \in [0,T],
\end{equation}
where I used $\mat{R}_\Gamma\, \mat{V}\,\bm{u}_{V} = \mat{V}_\Gamma \, \bm{u}_V$.  The initial condition for $\bm{u}_V$ is found by left multiplying \eqref{eq:mc_ic} by $\mat{V}^T \, \mat{W}$:
\begin{equation}\label{eq:uv_ic}
    \bm{u}_V(0) = \mat{V}^T\, \mat{W} \, \mat{V} \, \mat{V}^T \, \mat{W} \, \bm{u}_0 = \mat{V}^T \, \mat{W} \, \bm{u}_0.
\end{equation}
Comparing Equations~\eqref{eq:uv} and \eqref{eq:uv_ic} with Equations \eqref{eq:dg} and \eqref{eq:dg_ic}, respectively, we see that $\tilde{\bm{u}}$ and $\bm{u}_V$ are governed by identical initial value problems.  The modal DG semi-discretization is assumed to be well-posed, so its solution exists and is unique, which completes the proof.
\end{proof}


\begin{remark}
The proof can easily be generalized to multi-element meshes, provided the MC discretization uses the appropriate $\mat{R}_\Gamma$ matrices to project the solution to the face quadrature nodes.
\end{remark}

\subsection{Spectral equivalence of modal-collocation and modal DG}


Suppose the PDE in the IBVP \eqref{eq:hyperbolic} is linear.  Then the modal DG semi-discretization \eqref{eq:dg} can be written in the form 
\begin{equation}\label{eq:dg_linear}
\frac{d \tilde{\bm{u}}}{dt} = \tilde{\mat{A}} \tilde{\bm{u}} + \tilde{\bm{b}}(t),
\qquad\forall\, t \in [0,T],
\end{equation}
and the MC semi-discretization can be written as 
\begin{equation}\label{eq:mc_linear}
\frac{d \bm{u}}{dt} = \mat{A} \bm{u} + \bm{b}(t),
\qquad\forall\, t \in [0,T].
\end{equation}
One of the consequences of Theorem~\ref{thm:equivalence} is that spectra of the matrices $\tilde{\mat{A}}$ and $\mat{A}$ in \eqref{eq:dg_linear} and \eqref{eq:mc_linear} are effectively equivalent.

\begin{coro}\label{cor:spectrum}
The spectrum of the matrix $\mat{A}$ in the MC semi-discretization \eqref{eq:mc_linear} can be partitioned into two sets. The first set consists of the $N_P$ eigenvalues corresponding to the matrix $\tilde{\mat{A}}$ in the DG semi-discretization \eqref{eq:dg_linear}.  The second set consists of $N_Z = N - N_P$ repeated zero eigenvalues.
\end{coro}

\begin{proof}
Substitute the decomposition $\bm{u}(t) = \mat{V} \tilde{\bm{u}}(t) + \mat{Z} \bm{u}_Z(t)$ from the proof of Theorem~\ref{thm:equivalence} into \eqref{eq:mc_linear}:
\begin{equation*}
   \begin{bmatrix} \mat{V},\, \mat{Z} \end{bmatrix} \frac{d}{dt} \begin{bmatrix} \tilde{\bm{u}} \\ \bm{u}_Z \end{bmatrix}
    = \mat{A} \begin{bmatrix} \mat{V},\, \mat{Z} \end{bmatrix} \begin{bmatrix} \tilde{\bm{u}} \\ \bm{u}_Z \end{bmatrix}.
\end{equation*}
The inverse of the $N \times N$ matrix on the left-hand side is
\begin{equation*}
    \begin{bmatrix} \mat{V},\, \mat{Z} \end{bmatrix}^{-1}
    = \begin{bmatrix} \mat{V}^T\, \mat{W} \\ \mat{Z}^T\, \mat{W} \end{bmatrix}.
\end{equation*}
Left multiplying by this inverse, the ODE above becomes
\begin{align}
\frac{d}{dt} \begin{bmatrix} \tilde{\bm{u}} \\ \bm{u}_Z \end{bmatrix}
    &= \begin{bmatrix} \mat{V},\, \mat{Z} \end{bmatrix}^{-1} \mat{A} \begin{bmatrix} \mat{V},\, \mat{Z} \end{bmatrix} \begin{bmatrix} \tilde{\bm{u}} \\ \bm{u}_Z \end{bmatrix} \notag \\
    &= \underbrace{\begin{bmatrix} \mat{V}^T \,\mat{W}\, \mat{A}\, \mat{V} \quad & \mat{V}^T\, \mat{W}\, \mat{A}\, \mat{Z} \\[1.5ex] \mat{Z}^T\, \mat{W}\, \mat{A}\, \mat{V}\quad & \mat{Z}^T \,\mat{W}\, \mat{A}\, \mat{Z} \end{bmatrix}}_{\displaystyle \equiv \mat{B} } \begin{bmatrix} \tilde{\bm{u}} \\ \bm{u}_Z \end{bmatrix}. \label{eq:B}
\end{align}
Note that the matrix $\mat{B}$, defined above, is similar to $\mat{A}$, so the two matrices have the same eigenvalues, i.e.~$\sigma(\mat{A}) = \sigma(\mat{B})$.

Now, from the DG semi-discretization \eqref{eq:dg_linear} and the trivial ODE $d \bm{u}_Z/dt = \bm{0}$ --- the latter follows from Theorem~\ref{thm:equivalence} --- we also have 
\begin{equation*}
\frac{d}{dt} \begin{bmatrix} \tilde{\bm{u}} \\ \bm{u}_Z \end{bmatrix}
    = \begin{bmatrix} \tilde{\mat{A}} & \mat{0}_{N_P \times N_Z} \\[1.5ex] \mat{0}_{N_Z \times N_P} & \mat{0}_{N_Z \times N_Z} \end{bmatrix} \begin{bmatrix} \tilde{\bm{u}} \\ \bm{u}_Z \end{bmatrix}.
\end{equation*}
Equating the matrix on the right-hand side with the matrix $\mat{B}$ in \eqref{eq:B}, we can conclude that $\sigma(\mat{A}) = \sigma(\mat{B}) = \sigma(\tilde{\mat{A}}) \cup \{ 0 \}_{i=1}^{N_Z}$, as desired.
\end{proof}

Corollary~\ref{cor:spectrum} implies that MC solutions are effectively independent of the choice of quadrature.  In particular, the quadrature node distribution should have no impact on the spectral radius; I will illustrate this in Section~\ref{sec:results}.  The caveat suggested by the word ``effectively'' is that the projection is only approximate for functions that are not in $\mathbb{P}^{P}(\Omega)$, so the spectrum will differ between quadrature rules, in general; however, these differences will be on the order of the discretization.

\subsection{Nullspace inconsistency and steady problems}

A first-derivative operator $\mat{D}_{x_d}$ is \emph{nullspace consistent}~\cite{Svard2019convergence} if its nullspace consists only of constant vectors; that is, $\mat{D}_{x_d} \bm{u} = \bm{0}$ implies that $\bm{u}$ is a constant vector.  The operator is consistent in the sense that the differential operator $\partial/\partial_{x_d}$ also has only constant functions in its nullspace.

MC operators are \emph{nullspace inconsistent} whenever $N > N_P$.  This is easy to see using the matrix $\mat{Z} \in \mathbb{R}^{N \times N_Z}$, which was defined in the proof of Theorem~\ref{thm:mc_boundary} as the nullspace of $\mat{V}^T\, \mat{W}$.  It follows immediately that $\mat{D}_{x_d}\, \mat{Z} = \mat{V}_{x_d} \, \mat{V}^T \, \mat{W} \, \mat{Z} = \mat{0}_{N \times N_Z}$.

In the context of finite-difference discretizations, nullspace inconsistency can lead to inaccurate solutions~\cite{Svard2019convergence,Glaubitz2026SBPnotenough}.  However, the nullspace inconsistency of MC operators does not appear to present a problem, because the nullspace modes are initialized to zero and never grow; recall $d \bm{u}_Z /dt = \bm{0}$ in the proof of Theorem~\ref{thm:equivalence}.

\begin{remark}
The projected initial condition \eqref{eq:mc_ic} is responsible for setting the initial nullspace modes to zero.  However, projecting the initial condition is unnecessary for smooth problems, because the nullspace modes would correspond to $\text{O}(h^{P+1})$ errors that, while not damped, never grow.
\end{remark}

While not an issue for unsteady problems, the nullspace inconsistency of MC operators is problematic for steady problems, because Theorem~\ref{thm:equivalence} implies that the Jacobian of the right-hand side of \eqref{eq:mc_weak} will be singular.  Similarly, Corollary \ref{cor:spectrum} indicates that the system matrix for linear problems has zero eigenvalues and is therefore singular.

Fortunately, there is a simple solution to the nullspace inconsistency of MC discretization: local-projection stabilization (LPS).  LPS was first proposed in \cite{Becker2001finite} and adapted to multidimensional SBP discretizations in \cite{Hicken2020entropy}.

\begin{definition}[Local-projection stablization operator]\label{def:lps}
    Let $\{ X_{\Omega}, W_{\Omega}\}$ be a positive-interior quadrature rule that is exact for polynomials of total degree $2P$ on the domain $\Omega$.  Then the matrix $\mat{P} \in \mathbb{R}^{N\times N}$ is a degree $P$ local-projection stabilization operator if 
    \begin{align*}
        \mat{P} &= \big( \mat{I}_{N} - \mat{V} \mat{V}^T \mat{W} \big)^T \mat{W} \big( \mat{I}_{N} - \mat{V} \mat{V}^T \mat{W} \big) \\
        &= \mat{W} \big( \mat{I}_{N} - \mat{V} \mat{V}^T \mat{W} \big),
    \end{align*}
    where the second line follows from the orthonormality of $\mat{V}$ with respect to $\mat{W}$.
\end{definition}

To remove the nullspace from the MC semi-discretizations, one subtracts $\mat{P} \bm{u}$ from the right-hand side of the weak formulation \eqref{eq:mc_weak}, or $\mat{W}^{-1} \mat{P} \bm{u}$ from the right-hand side of the strong formulation \eqref{eq:mc}.

LPS does not affect polynomials in $\mathbb{P}^{P}(\Omega)$, since 
\begin{equation*}
    \mat{P} \mat{V} = \mat{W} \mat{V} - \mat{W} \mat{V} \big( \underbrace{\mat{V}^T \mat{W} \mat{V}}_{\displaystyle \mat{I}_{N_P}} \big) = \mat{0}_{N \times N_P},
\end{equation*}
so it does not influence the polynomial exactness of the rest of the discretization.  By contrast, LPS is nonzero when applied to $\mat{Z}$, since $\mat{P} \mat{Z} = \mat{W} \mat{Z}$.  In addition to accuracy, one can show that LPS is conservative and energy/entropy stable~\cite{Hicken2020entropy}.

\begin{remark}\label{rmk:lps_scaling}
In practice, LPS is modified to make it dimensionally consistent and scaled to the magnitude of the local wave speed(s).  This can be achieved by inserting a diagonal, positive definite matrix --- or block-diagonal, positive-definite matrix for PDE systems --- between the projections and the norm $\mat{W}$.  If $\mat{\Lambda} \in \mathbb{R}^{N\times N}$ is such a positive-definite matrix then the modified LPS operator is
\begin{equation}\label{eq:lps_scaled}
    \mat{P} =  \big( \mat{I}_{N} - \mat{V} \mat{V}^T \mat{W} \big)^T \mat{W} \mat{\Lambda} \big( \mat{I}_{N} - \mat{V} \mat{V}^T \mat{W} \big)^T.
\end{equation}
Since $\mat{\Lambda}$ is positive definite, including this factor does not impact the stability of LPS.  Moreover, polynomial exactness and conservation are maintained.
\end{remark}

\begin{remark}
Consider the first-derivative operators 
\begin{equation*}
    \mat{D}_{x_d}^{(+)} = \mat{H}^{-1} \big( \mat{Q}_{x_d} + \mat{P} \big),
    \qquad\text{and}\qquad
    \mat{D}_{x_d}^{(-)} = \mat{H}^{-1} \big( \mat{Q}_{x_d} - \mat{P} \big).
\end{equation*}
It is straightforward to show that $\mat{D}_{x_d}^{(+)}$ and $\mat{D}_{x_d}^{(-)}$ are diagonal-norm upwind SBP operators~\cite{Mattsson2017upwindSBP}.  An attractive feature of upwind SBP operators is that they can be combined with flux-vector splitting to determine an appropriate scaling matrix $\mat{\Lambda}$ in\eqref{eq:lps_scaled}. See, for example, \cite{Lundgren2020efficient} and \cite{Hew2025strongly} for examples of upwind SBP operators in the context of finite-difference methods.
\end{remark}

\subsection{Generalizations}\label{sec:generalizations}

I chose to use an orthonormal basis and a degree $2P$ quadrature to simplify the presentation, but the theory can be adapted if we use a nonorthogonal basis for $\mathbb{P}^{P}(\Omega)$ and a degree $2P-1$ quadrature.

For a nonorthogonal basis $\{ \hat{\fnc{V}}_i \}_{i=1}^{N_P}$ and $Q \geq 2P -1$ exact quadrature, the MC operator becomes~\cite{Chan2018discretely}
\begin{equation}\label{eq:mc_nonortho}
    \mat{D}_{x_d} = \hat{\mat{V}}_{x_d} \, \mat{M}^{-1} \, \hat{\mat{V}}^T \,\mat{W},
\end{equation}
where $\mat{M} \equiv \hat{\mat{V}}^T \, \mat{W} \, \hat{\mat{V}}$ is the mass matrix, and $\hat{\mat{V}}$ and $\hat{\mat{V}}_{x_d}$ hold the basis and its derivative evaluated at the quadrature nodes; these matrices are analogous to $\mat{V}$ and $\mat{V}_{x_d}$ in the orthonormal case.  In addition, the boundary operator remains $\mat{E}_{x_d} = \mat{R}_{\Gamma}^T \, \mat{W}_{\Gamma} \, \mat{N}_{x_d} \, \mat{R}_{\Gamma}$, but the boundary-projection operator is now defined by~\cite{Chan2018discretely}
\begin{equation}\label{eq:proj_nonortho}
    \mat{R}_{\Gamma} = \hat{\mat{V}}_{\Gamma} \, \mat{M}^{-1} \, \hat{\mat{V}}^T \, \mat{W}.
\end{equation}

\paragraph{Nonorthogonal basis with $2P$ exact quadrature}

If the volume quadrature remains $2P$ exact, then the MC operator \eqref{eq:mc_nonortho} is equivalent to one based on an orthonormal basis.  To show this, I use the Cholesky decomposition of the symmetric, positive definite mass matrix, i.e.~$\mat{M} = \mat{L} \mat{L}^T$.  Then
\begin{equation*}
    \mat{D}_{x_d} = \hat{\mat{V}}_{x_d} \mat{L}^{-T} \mat{L}^{-1} \hat{\mat{V}}^T \mat{W} = \mat{V}_{x_d} \mat{V}^{T} \mat{W},
\end{equation*}
where I have made the identifications $\mat{V} = \hat{\mat{V}} \mat{L}^{-T}$ and $\mat{V}_{x_d} = \hat{\mat{V}}_{x_d} \mat{L}^{-T}$.  It is then easy to verify that $\mat{V}$ is orthonormal with respect to $\mat{W}$:
\begin{equation*}
    \mat{V}^T \mat{W} \mat{V} = \mat{L}^{-1} \hat{\mat{V}}^T \mat{W} \hat{\mat{V}} \mat{L}^{-T} = \mat{L}^{-1} \mat{L} \mat{L}^T \mat{L}^{-T} = \mat{I}_{N_P}.
\end{equation*}
In other words, the role of the (inverse) mass matrix in \eqref{eq:mc_nonortho} and \eqref{eq:proj_nonortho} is to transform the nonorthogonal basis into an orthonormal one.  Thus, the theory developed for the orthonormal basis applies to the nonorthogonal case by recognizing that there is an implicit orthonormal basis defined by $\{ \fnc{V}_i = \sum_{j=1}^{i} [\mat{L}^{-1}]_{ij} \hat{\fnc{V}}_j \}_{i=1}^{N_P}$ and because $\mat{L}$ is independent of the quadrature when $Q=2P$.

\paragraph{Degree $2P-1$ volume quadrature}

The theory for first-derivative SBP operators tells us that we need only $2P-1$ exactness from the volume quadrature.  This is also true for MC operators.  In the following list, I explain how the various results must be adapted.

\begin{description}
\item[Theorem \ref{thm:mc_are_sbp}:] Consider the MC operator defined by \eqref{eq:mc_nonortho}, but using a degree $2P-1$ exact volume quadrature.  It is straightforward to show that such an operator is also a diagonal-norm SBP operator by following the same steps as used in the proof of Theorem~\ref{thm:mc_are_sbp}.

\item[Proposition \ref{prop:convert}:] If $\mat{D}_{x_d} \in \mathbb{R}^{N\times N}$ is a degree $P$ first-derivative operator that acts at the nodes of a degree $2P-1$ quadrature rule, then $\hat{\mat{V}} \mat{M}^{-1} \hat{\mat{V}}^T \mat{W} \mat{D}_{x_d} \hat{\mat{V}} \mat{M}^{-1} \hat{\mat{V}}^T \mat{W}$ is a degree $P$ modal collocation operator of the form~\eqref{eq:mc_nonortho}.

\item[Lemma \ref{thm:mc_boundary}:] A degree $2P-1$ volume quadrature was already considered in Lemma 1 of \cite{Chan2018discretely}.  Indeed, to adapt Lemma~\ref{thm:mc_boundary}, we simply replace the boundary-projection operator with the one defined by~\eqref{eq:proj_nonortho}.  Furthermore, note that compatibility, Definition~\ref{def:compatibility}, can be satisfied with a degree $2P-1$ exact volume quadrature and $2P$ exact boundary quadrature, at least for elements with piecewise linear facets.

\item[Theorem \ref{thm:equivalence}:] The main result continues to hold provided the source and initial condition are projected using $\hat{\mat{V}} \mat{M}^{-1} \hat{\mat{V}}^T \mat{W}$ rather than $\mat{V} \mat{V}^T \mat{W}$.  The matrix $\mat{Z}$ is now the nullspace basis for $\hat{\mat{V}}^T \mat{W}$, but it is still orthonormalized with respect $\mat{W}$.

\item[Corollary \ref{cor:spectrum}:] The spectrum of the MC semi-discretization is partitioned in the same manner described in Corollary \ref{cor:spectrum}, provided the DG semi-discretization uses the same $2P-1$ exact volume quadrature as the MC semi-discretization.  Thus, the nontrivial eigenvalues of $\mat{A}$ continue to be the eigenvalues of the DG matrix $\tilde{\mat{A}}$.
\end{description}

\paragraph{Nonpositive quadrature weights}

The proof of Theorem~\ref{thm:equivalence} did not explicitly require the quadrature rule to be a positive-interior rule.  This points to an interesting result: if the modal DG discretization is energy stable, then the MC discretization is also energy stable, \emph{even if it uses a volume quadrature rule with negative weights}.  I will demonstrate this in Section~\ref{sec:results}.

\section{Numerical experiments}\label{sec:results}

\subsection{MC and SBP operator construction}



I use triangle elements for the numerical experiments.  Triangular and tetrahedral meshes are often used for complex geometries that arise in engineering applications.  Consequently, there has been considerable interest in developing discretizations that are suitable for simplex elements, including in the SBP literature~\cite{Hicken2016multidimensional,Crean2018entropy,Chan2018discretely,Worku2025tensor,Worku2025veryhigh,Marchildon2020optimization,Montoya2024efficient}.

The starting point for the MC and SBP triangle operators is a commonly used (degenerate) mapping from the square domain $\Omega_{\xi} = [-1,1]^2$ to the reference right triangle~\cite{Dubiner1991spectral}.
\begin{equation*}
    \bm{x}(\bm{\xi}) = \begin{bmatrix} x_1(\bm{\xi}) \\ x_2(\bm{\xi}) \end{bmatrix} = \begin{bmatrix} \frac{1}{2}(1 + \xi_1)(1 - \xi_2) - 1 \\ \xi_2 \end{bmatrix},
\end{equation*}
where $\bm{\xi} = \big[ \xi_1, \xi_2\big]^T \in \Omega_{\xi}$ are the reference domain coordinates.  Following~\cite{Montoya2024efficient}, I use a tensor-product grid consisting of $Q/2+1$ Legendre-Gauss (LG) nodes\footnote{I use even $Q$ for all cases considered here, so $Q/2$ is an integer.} in both the $\xi_1$ and $\xi_2$ coordinate directions, which yields a volume quadrature that is $2Q+1$ exact in $\bm{\xi}$ space. The additional ``$+1$'' degree of accuracy in the quadrature rule is needed for the Jacobian of the mapping.  The boundary quadrature also uses $Q/2+1$ nodes along the three edges of the triangle, which allows the min-norm SBP operators (described below) to use sparse reconstruction along collinear volume nodes.  Figure~\ref{fig:triangle_nodes} illustrates the tensor-product nodes for $Q=4$, $Q=8$ and $Q=12$ quadrature rules.

\begin{remark}
    The nodes clustering near the top of the triangles in Figure~\ref{fig:triangle_nodes} are responsible for the large spectral radii typically observed for nodal spectral-element methods on triangles. As the results demonstrate, MC operators are not impacted by this clustering.
\end{remark}

\begin{figure}[tbp]
\includegraphics[width=0.32\textwidth]{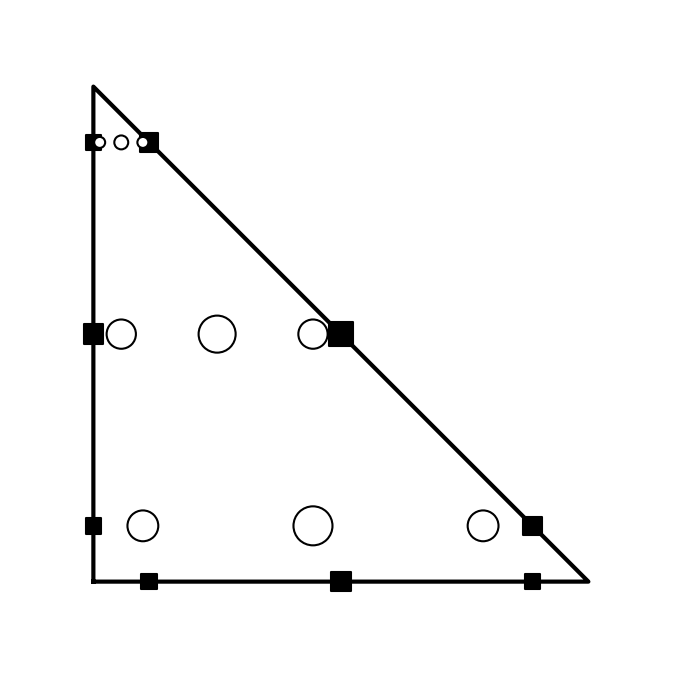}%
\includegraphics[width=0.32\textwidth]{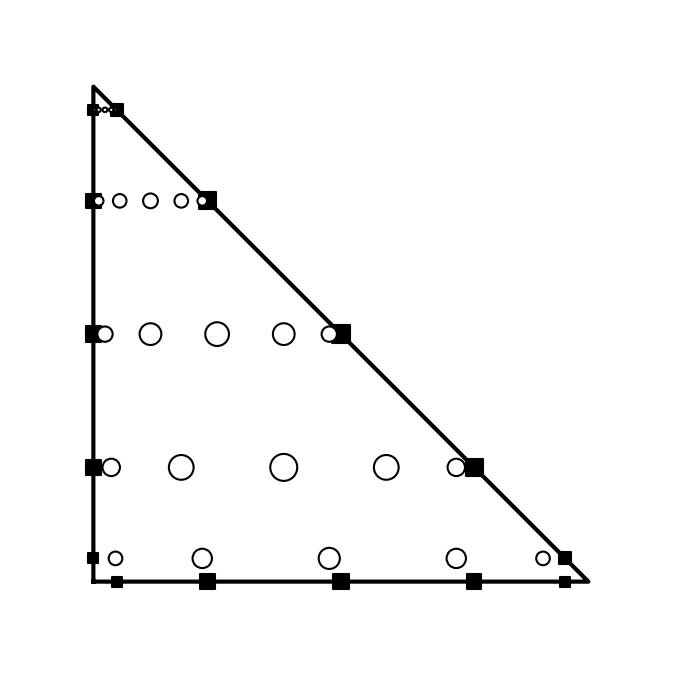}%
\includegraphics[width=0.32\textwidth]{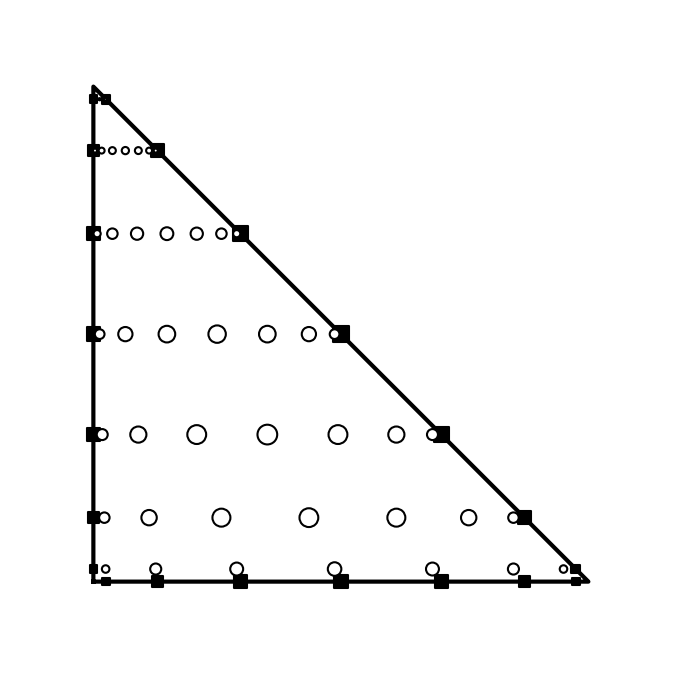}\\%
\caption{Quadrature nodes based on $Q=4$ (left), $Q=8$ (center), and $Q=12$ (right) for the triangle.  These quadratures have $N=9$, $N=25$, and $N=49$ solution nodes, respectively.  The white circles are the solution nodes and the black squares are the boundary quadrature nodes.  The size of a node is proportional to the magnitude of its quadrature weight. \label{fig:triangle_nodes}}
\end{figure}

The MC operators are constructed following Definition~\ref{def:mc} with the Proriol-Koornwinder-Dubiner polynomials~\cite{Proriol1957famille,Koornwinder1975two,Dubiner1991spectral} providing the orthonormal basis $\{\fnc{V}_{i}\}_{i=1}^{N_P}$.

The SBP operators used for comparison have their skew-symmetric matrix $\mat{S}_{x_d} \equiv \big( \mat{Q}_{x_d} - \mat{Q}_{x_d}^T\big)/2$ constructed using Equation (9) from~\cite{Hicken2025constructing}.  I verified numerically that the resulting $\mat{S}_{x_d}$ matrices are equivalent to solving the accuracy condition $\mat{S}_{x_d} \mat{V} = \mat{W} \mat{V}_{x_d} - \frac{1}{2} \mat{E}_{x_d} \mat{V}$ in a minimum-norm sense; hence, I refer to these as min-norm SBP operators.  I constructed the symmetric matrices $\mat{E}_{x_d}$ by using extrapolation operators along collinear volume nodes, thus ensuring that the boundary operators are sparse.

\subsection{Constant-coefficient advection}

\paragraph{IBVP and discretizations}

The first study is based on the following constant-coefficient advection problem:
\begin{equation}\label{eq:advection}
\begin{alignedat}{2}
    \frac{\partial \fnc{U}}{\partial t} &= -\sum_{d=1}^{D} \alpha_d \frac{\partial \fnc{U}}{\partial x_d},&\qquad &\forall\, \bm{x} \in \Omega, t \in [0,2], \\
    \fnc{U}(\bm{x},t) &= \bm{G}(\bm{x},t), &\qquad&\forall\, \bm{x} \in \Gamma^{-}, t \in [0,2], \\
    \fnc{U}(\bm{x},0) &= \bm{G}(\bm{x},0), &\qquad&\forall\, \bm{x} \in \Omega,
\end{alignedat}
\end{equation}
where, for simplicity, the spatial domain is the right triangle
\begin{equation*}
\Omega = \{ \bm{x} \in \mathbb{R}^2 \,|\, -1 \leq x_1 \leq 1, -1 \leq x_2 \leq -x_1\},
\end{equation*}
which is identical to the domain of the elements shown in Figure~\ref{fig:triangle_nodes}.  The advection velocity is given by $\bm{\alpha} = \big[\alpha_1, \alpha_2\big]^T = \big[1,1\big]^T/\sqrt{2}$, so the flow direction is $45^\circ$ counterclockwise from the $x_1$ axis and perpendicular to the hypotenuse.  Let $\alpha_n(\bm{x}) = \sum_{d=1}^{D} \alpha_d n_{x_d}(\bm{x})$ denote the normal component of the advection velocity at $\bm{x}$.  Then the inflow boundary 
\begin{equation*}
    \Gamma^{-} = \{ \bm{x} \in \Gamma \,|\, \alpha_n(\bm{x}) \leq 0 \},
\end{equation*}
consists of the bottom and left edges of the right triangle.  The boundary and initial conditions are defined using the exact solution
\begin{equation*}
    \fnc{G}(\bm{x}, t) = \sin \Big( 2\pi \big(  (x_1 + x_2)/\sqrt{2}  - t \big) \Big), \qquad \forall\, \bm{x} \in \Omega, t \in [0,2].
\end{equation*}

The MC and SBP semi-discretization of \eqref{eq:advection} can be inferred from \eqref{eq:mc} by replacing $\bm{f}_{d}(\bm{u})$ with $\alpha_{d} \bm{u}$ and eliminating the source term:
\begin{equation}\label{eq:advection_dis}
    \begin{alignedat}{2}
    \frac{d \bm{u}}{dt} &= -\sum_{d=1}^D \alpha_d \mat{D}_{x_d} \, \bm{u} - \mat{W}^{-1} \, \mat{R}_{\Gamma}^T \, \mat{W}_{\Gamma} \, \big( \bm{f}_n^*(\mat{R}_\Gamma \bm{u}, \bm{g}_\Gamma) - \bm{f}_n(\bm{u}) \big),&\qquad &\forall\, t \in [0,2], \\
    \bm{u}(0) &= \mat{V}\, \mat{V}^T \, \mat{W}\, \bm{g}(0),&&
    \end{alignedat}
\end{equation}
where the initial condition is based on the projected exact solution, $\fnc{G}(\bm{x},0)$.  The vector $\bm{f}_n(\bm{u})$ is defined by Equation~\eqref{eq:fn}, and the numerical flux function $\bm{f}_n^*(\mat{R}_\Gamma \bm{u}, \bm{g}_\Gamma)$ is defined at each boundary quadrature node $\hat{\bm{x}}_m \in X_{\Gamma}$ as
\begin{equation*}
    \big[ \bm{f}_n^*(\mat{R}_\Gamma \bm{u}, \bm{g}_\Gamma) \big]_m =
    \begin{cases} 
    \alpha_n(\hat{\bm{x}}_m) g_m(t), & \text{if } \alpha_n(\hat{\bm{x}}_m) \leq 0,\\
    \alpha_n(\hat{\bm{x}}_m) u_m(t), & \text{if } \alpha_n(\hat{\bm{x}}_m) > 0,
    \end{cases}
\end{equation*}
where $g_m(t) = [\bm{g}(t)]_m = \fnc{G}(\hat{\bm{x}}_m,t)$ is the boundary data evaluated at $\hat{\bm{x}}_m$, and $u_m(t) = [\mat{R}_\Gamma \bm{u}(t)]_m$ is the numerical solution projected to $\hat{\bm{x}}_m$.  Using this definition of the numerical flux, I rearrange the boundary terms in \eqref{eq:advection_dis} into a term that depends on the boundary data and a term that depends on the numerical solution:
\begin{equation}\label{eq:boundary_partition}
    \mat{W}^{-1} \, \mat{R}_{\Gamma}^T \, \mat{W}_{\Gamma} \, \big( \bm{f}_n^*(\mat{R}_\Gamma \bm{u}, \bm{g}_\Gamma) - \bm{f}_n(\bm{u}) \big)
    = \mat{W}^{-1} \, \mat{R}_{\Gamma}^T \, \mat{W}_{\Gamma} \mat{N}_{-} \bm{g}(t) + \mat{W}^{-1} \, \mat{R}_{\Gamma}^T \, \mat{W}_{\Gamma} \mat{N}_{+} \mat{R}_{\Gamma} \bm{u}(t),
\end{equation}
where I have introduced the diagonal matrices $[\mat{N}_{-}]_{mm} = \min\big(0,\alpha_n(\hat{\bm{x}}_m)\big)$ and $[\mat{N}_{+}]_{mm} = \max\big(0,\alpha_n(\hat{\bm{x}}_m)\big)$, and I used the identity $\sum_{d=1}^D \alpha_d \mat{N}_{x_d} = \mat{N}_{-} + \mat{N}_{+}$ to decompose $\bm{f}_n(\bm{u})$.

The spectrum for the semi-discretization \eqref{eq:advection_dis} requires writing the equation in the form $d\bm{u}/dt = \mat{A} \bm{u} + \bm{b}$. To this end, I substitute \eqref{eq:boundary_partition} into \eqref{eq:advection_dis} and find
\begin{equation}\label{eq:advection_semidis}
\frac{d \bm{u}}{dt} = \underbrace{\Bigg( -\sum_{d=1}^{D} \alpha_d \mat{D}_{x_d} + 
\mat{W}^{-1} \mat{R}_{\Gamma}^T \mat{W}_\Gamma \mat{N}_{-} \mat{R}_{\Gamma} \Bigg)}_{\displaystyle \equiv \mat{A}}\bm{u} \;\; \underbrace{\rule[-3ex]{0ex}{3ex} - \mat{W}^{-1} \mat{R}_{\Gamma}^T \mat{W}_{\Gamma} \mat{N}_{-} \bm{g}(t)}_{\displaystyle \equiv \bm{b}(t)}.
\end{equation}

The modal DG discretization is obtained by substituting $\bm{u}(t) = \mat{V}\, \tilde{\bm{u}}(t)$ into Equation \eqref{eq:advection_semidis} and left multiplying the equation by $\mat{V}^T \mat{W}$.

\paragraph{Verification of MC-DG equivalence}

I solve the MC and modal DG initial-value problems using the five-stage, fourth order low-storage method of Carpenter and Kennedy~\cite{Carpenter1994fourth} as implemented in the Julia \texttt{DifferentialEquations.jl} library~\cite{Rackauckas2017differentialequations}.  To ensure that the temporal discretizations are equivalent, I use a fixed step size of $\Delta t = 2/(P+1)^2$.  

Table \ref{tab:l2diff_advection} lists the differences between the MC and modal DG solutions at the terminal time $T=2$ for degrees $P=3, 6, 9$ and 12.  The difference is evaluated using the quadrature-based approximation of the $L^2$ norm:
\begin{equation}\label{eq:l2_diff}
    L^2~\text{Diff.} = \sqrt{(\mat{V}\tilde{\bm{u}} - \bm{u})^T \mat{W} (\mat{V}\tilde{\bm{u}} - \bm{u})}.
\end{equation}
For each of the four polynomial degrees considered, the table shows the differences obtained using quadratures that are $2p$, $4p$, and $6p$ exact.  As expected from Theorem~\ref{thm:equivalence}, the MC and modal DG solutions are equivalent to machine precision.

\begin{table}[tbp]
\centering
\caption{$L^2$ difference between MC and modal DG solutions to the constant-coefficient advection problem on a triangular domain. \label{tab:l2diff_advection}}
\begin{tabular}{clll}
\rule{0ex}{5ex}
 & \multicolumn{3}{c}{\textbf{quadrature exactness}} \\ \cmidrule{2-4}
\textbf{degree $P$} & \multicolumn{1}{c}{$2P$} & \multicolumn{1}{c}{$4P$} & \multicolumn{1}{c}{$6P$} \\ \hline
\rule{0ex}{3ex}%
3 & 6.6533e-16 & 3.2709e-16 & 5.0756e-16\\[1ex]
6 & 1.0038e-15 & 3.4708e-15 & 2.2246e-15\\[1ex]
9 & 2.5324e-15 & 2.6853e-15 & 4.4794e-15\\[1ex]
12 & 4.7917e-15 & 2.4298e-15 & 4.1976e-15\\[1ex]\hline
\end{tabular}
\end{table}

\paragraph{Comparison of MC and SBP spectra}

Figures~\ref{fig:spectrum_p3} and \ref{fig:spectrum_p6} present representative spectra for the matrix $\mat{A}$ in the semi-discretization~\eqref{eq:advection_semidis}.  The plots in Figure~\ref{fig:spectrum_p3} show the spectra for $P=3$ when using quadratures with $Q=6$, $Q=12$, and $Q=18$, which correspond to MC and SBP operators with $N=16$, $N=49$, and $N=100$ nodes, respectively.  Similarly, Figure~\ref{fig:spectrum_p6} displays the spectra for $P=6$ with $Q=12$, $Q=24$, and $Q=36$; these sixth-order operators have $N=49$, $169$, and $361$ nodes, respectively.

From Corollary 1, a degree $P$ MC-based discretization should have a spectrum that matches the spectrum of its associated modal DG discretization, notwithstanding the $N-N_P$ repeated zero eigenvalues of the MC discretization.  This result is verified by the (upper) plots in Figures~\ref{fig:spectrum_p3} and \ref{fig:spectrum_p6}.  In this particular case, the nonzero eigenvalues are independent of $Q \geq 2P$, since the IBVP has a constant-coefficient advection field and no mapping is applied to the reference element.

Contrast the MC eigenvalues with those of the min-norm SBP operator; in particular, compare the range for the real axes.  The catastrophic growth in the SBP spectral radius has two causes.  First and most significantly, the SBP operators are impacted by the clustering of the nodes at the top of the triangles; see Figure~\ref{fig:triangle_nodes}.  This effect is not unique to SBP operators and has long been known in the spectral element community~\cite{Dubiner1991spectral}.  The second cause is the min-norm solution to the SBP accuracy conditions, which does not control the spectral radius.

\begin{figure}[tbp]
\includegraphics[width=0.32\textwidth]{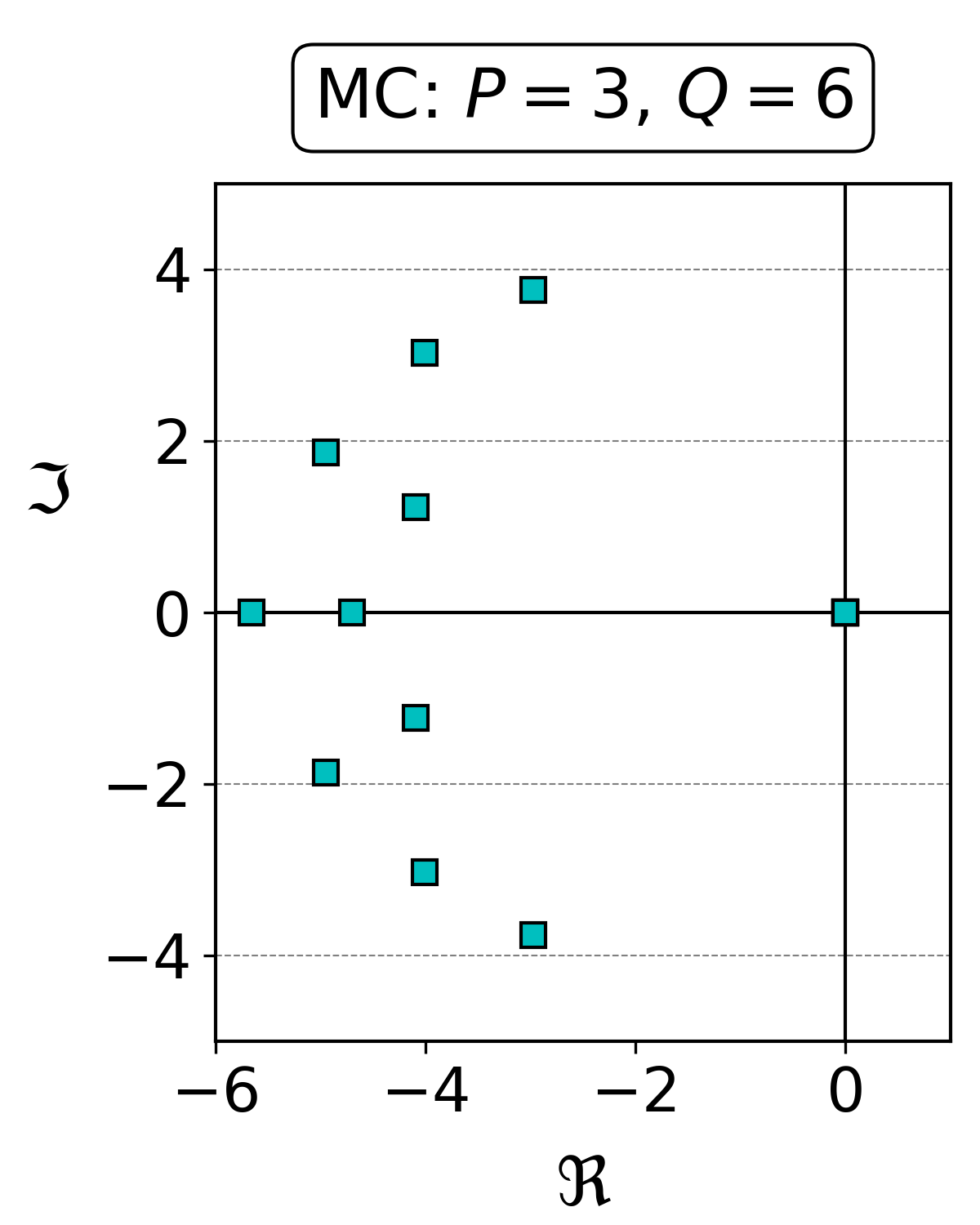}%
\includegraphics[width=0.32\textwidth]{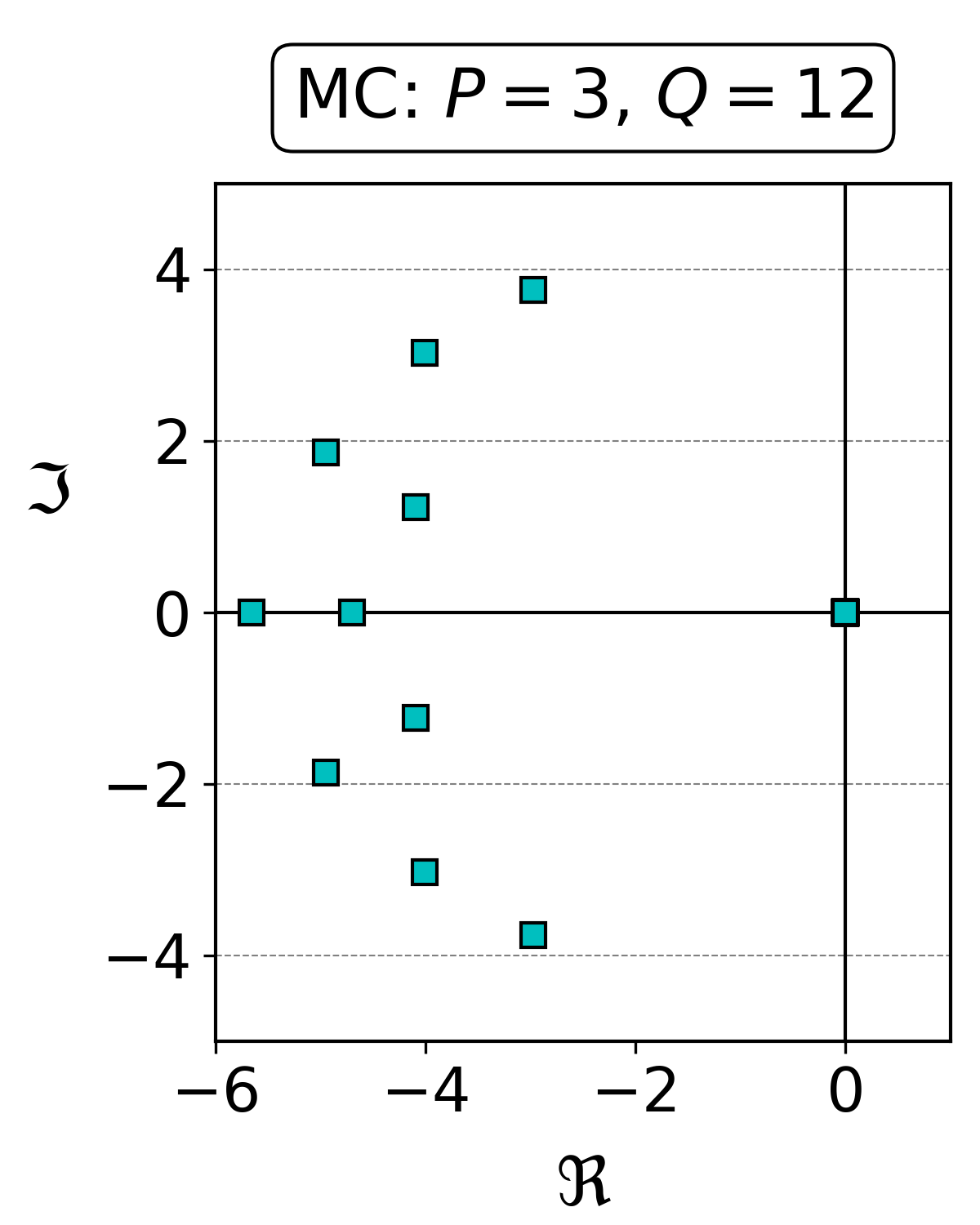}%
\includegraphics[width=0.32\textwidth]{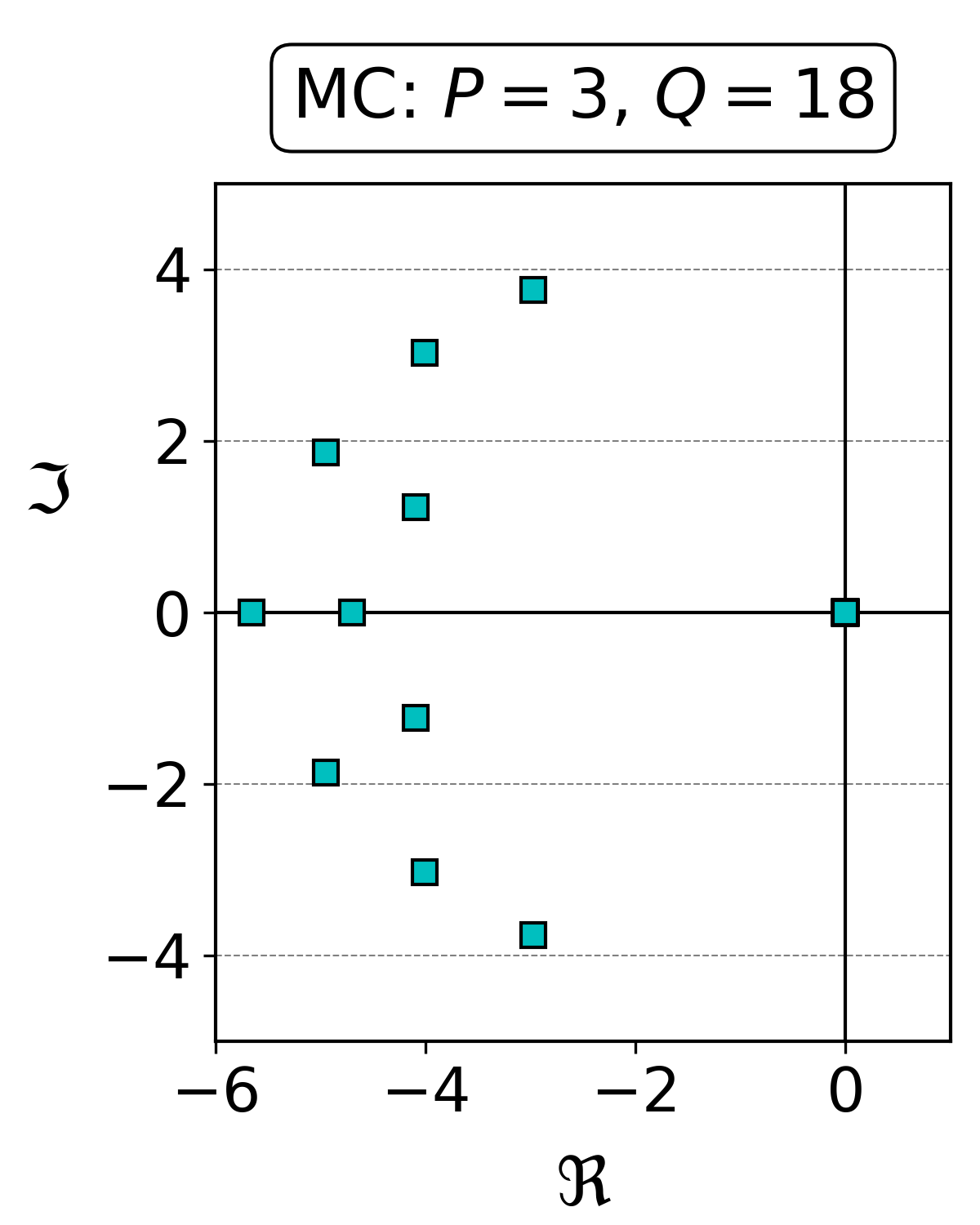}\\%
\includegraphics[width=0.32\textwidth]{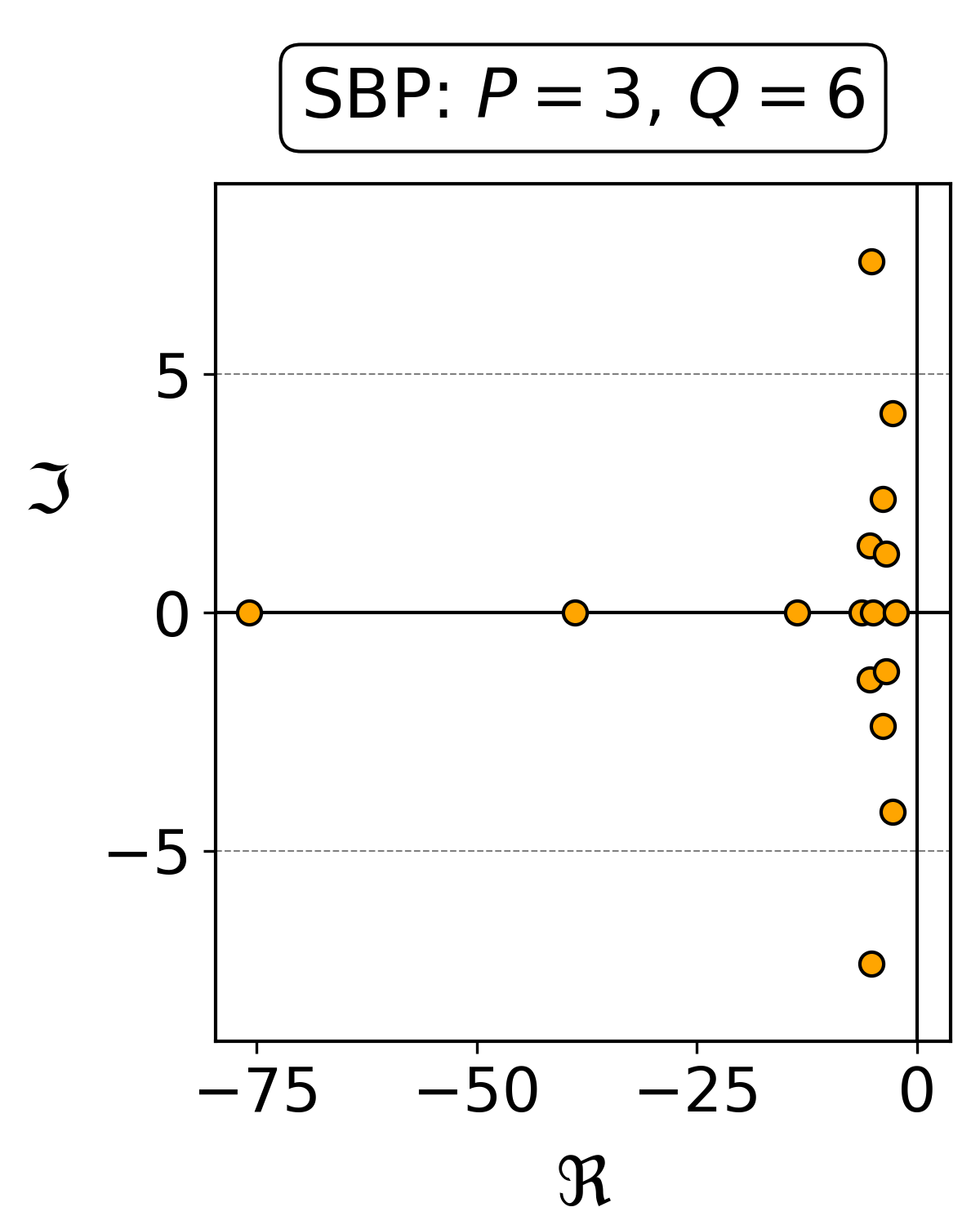}%
\includegraphics[width=0.32\textwidth]{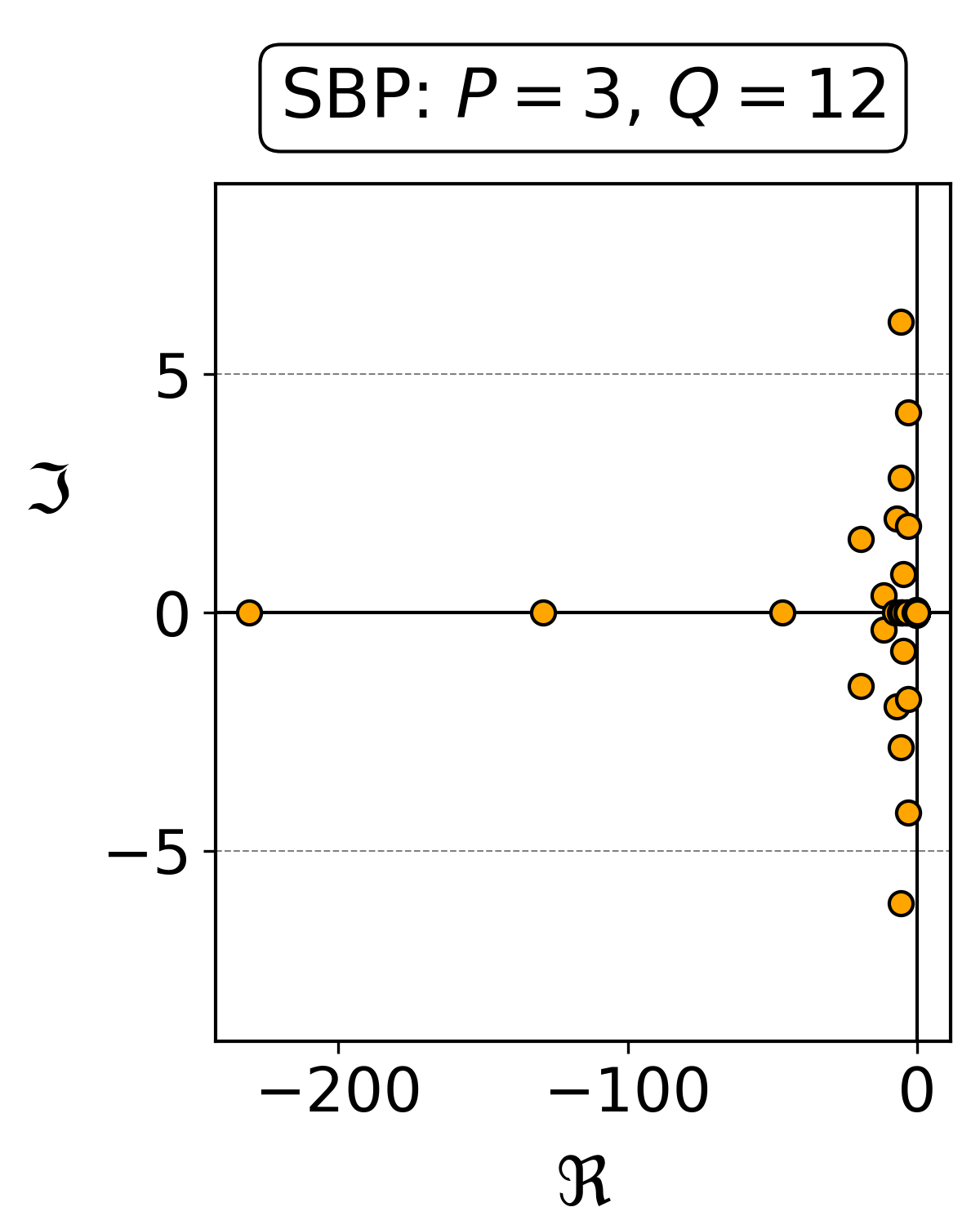}%
\includegraphics[width=0.32\textwidth]{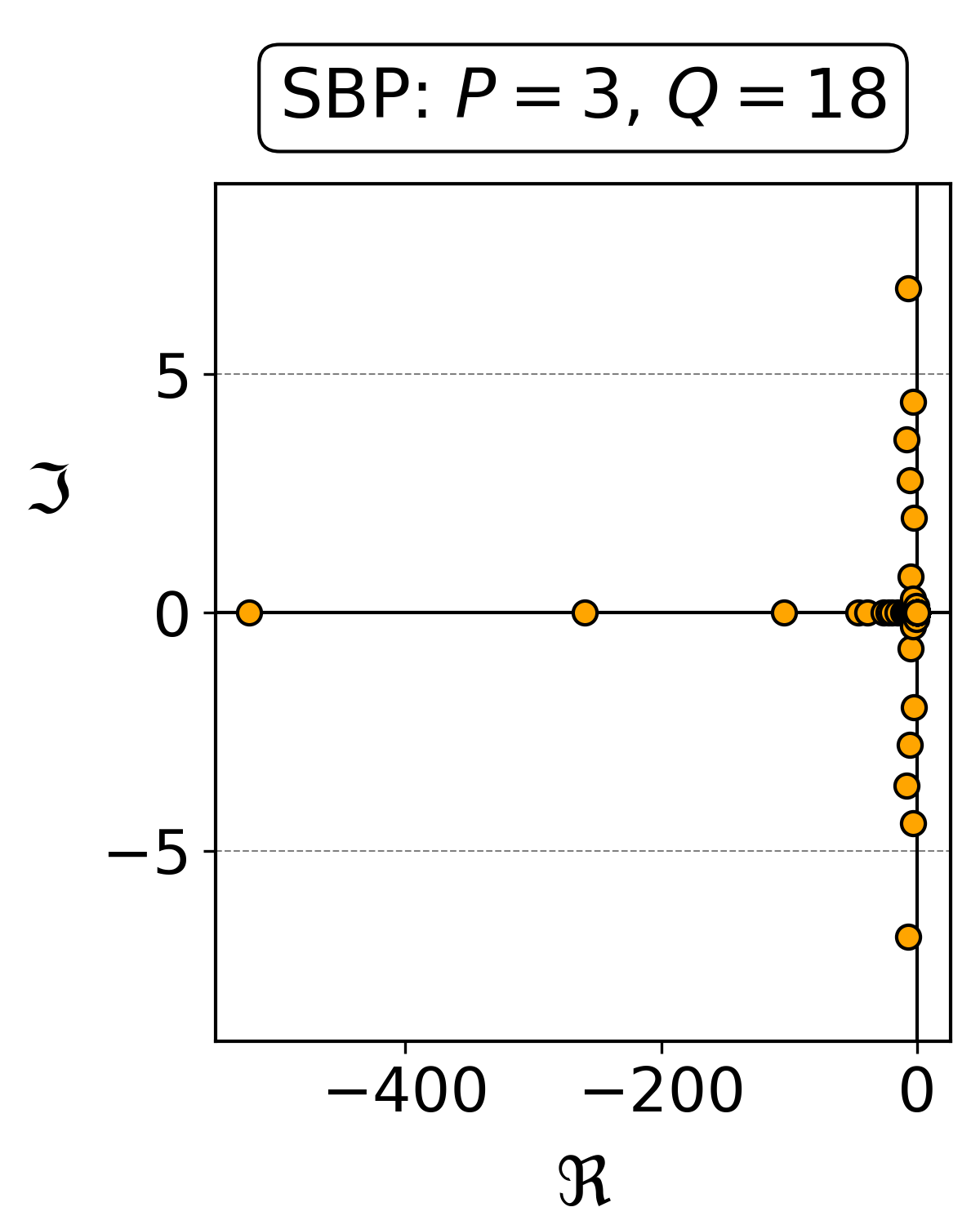}\\%
\caption{Eigenvalue distributions for the triangle-advection problem using $P=3$ MC (top) and SBP (bottom) operators with different quadrature degrees $Q$. \label{fig:spectrum_p3}}
\end{figure}

\begin{figure}[tbp]
\includegraphics[width=0.32\textwidth]{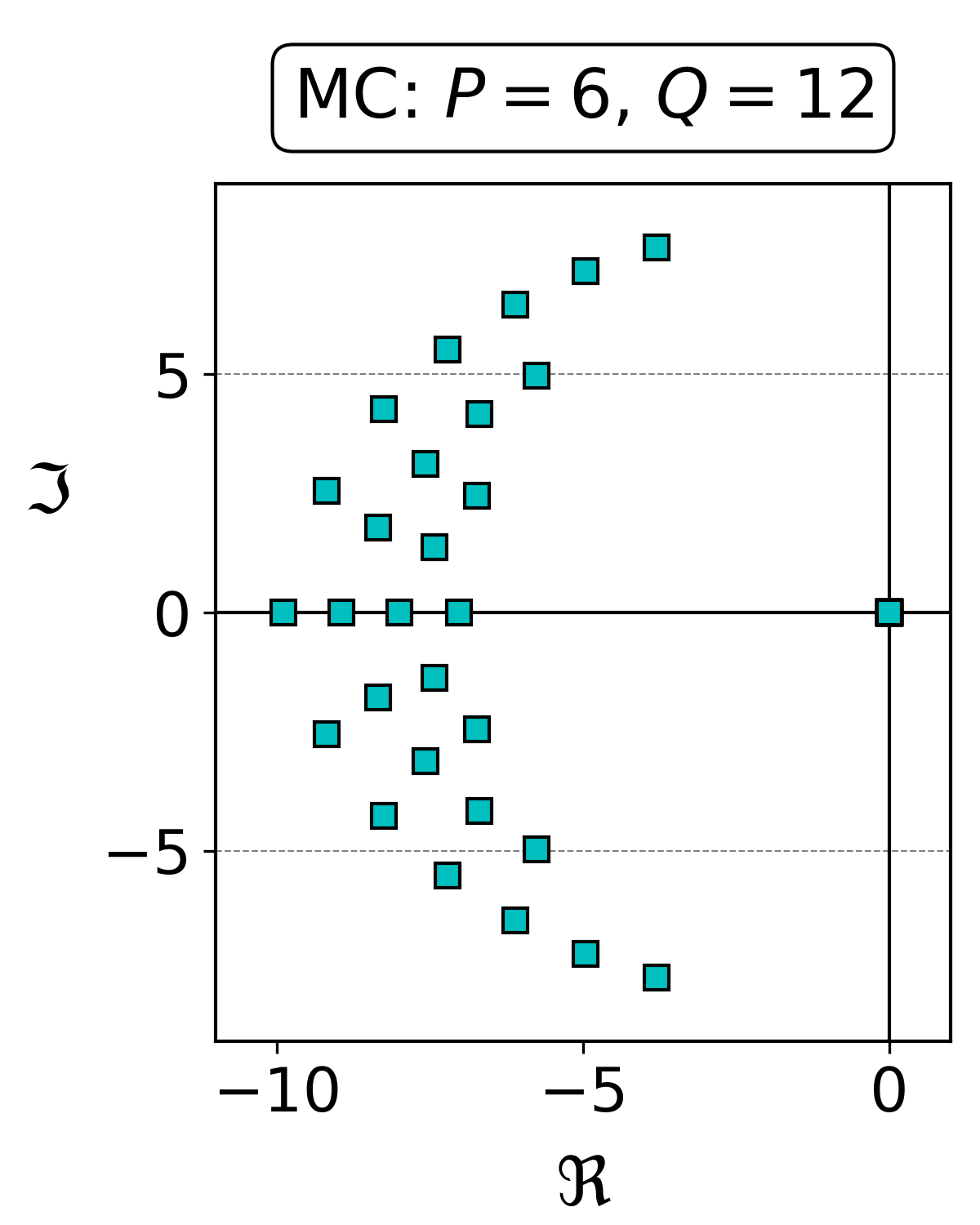}%
\includegraphics[width=0.32\textwidth]{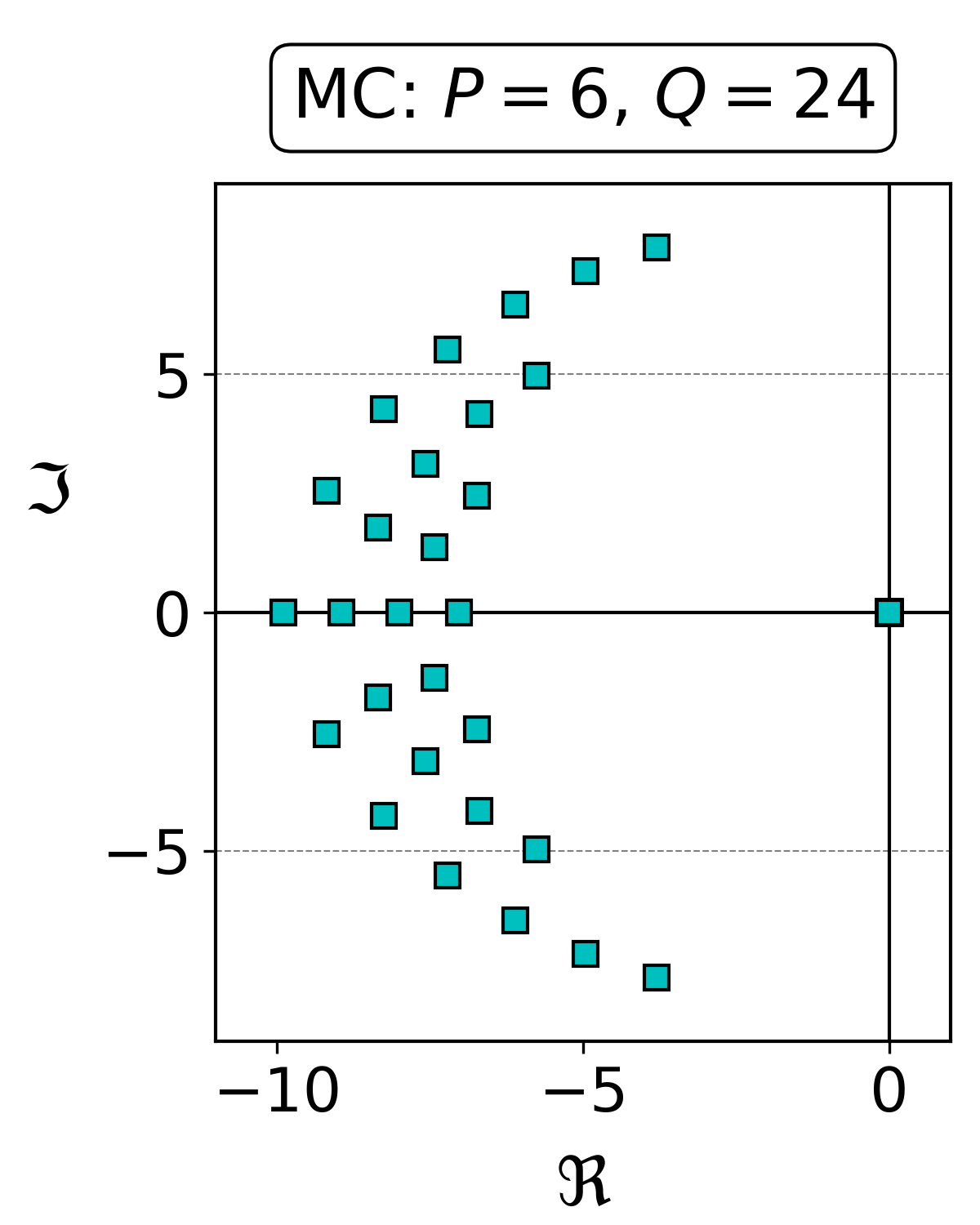}%
\includegraphics[width=0.32\textwidth]{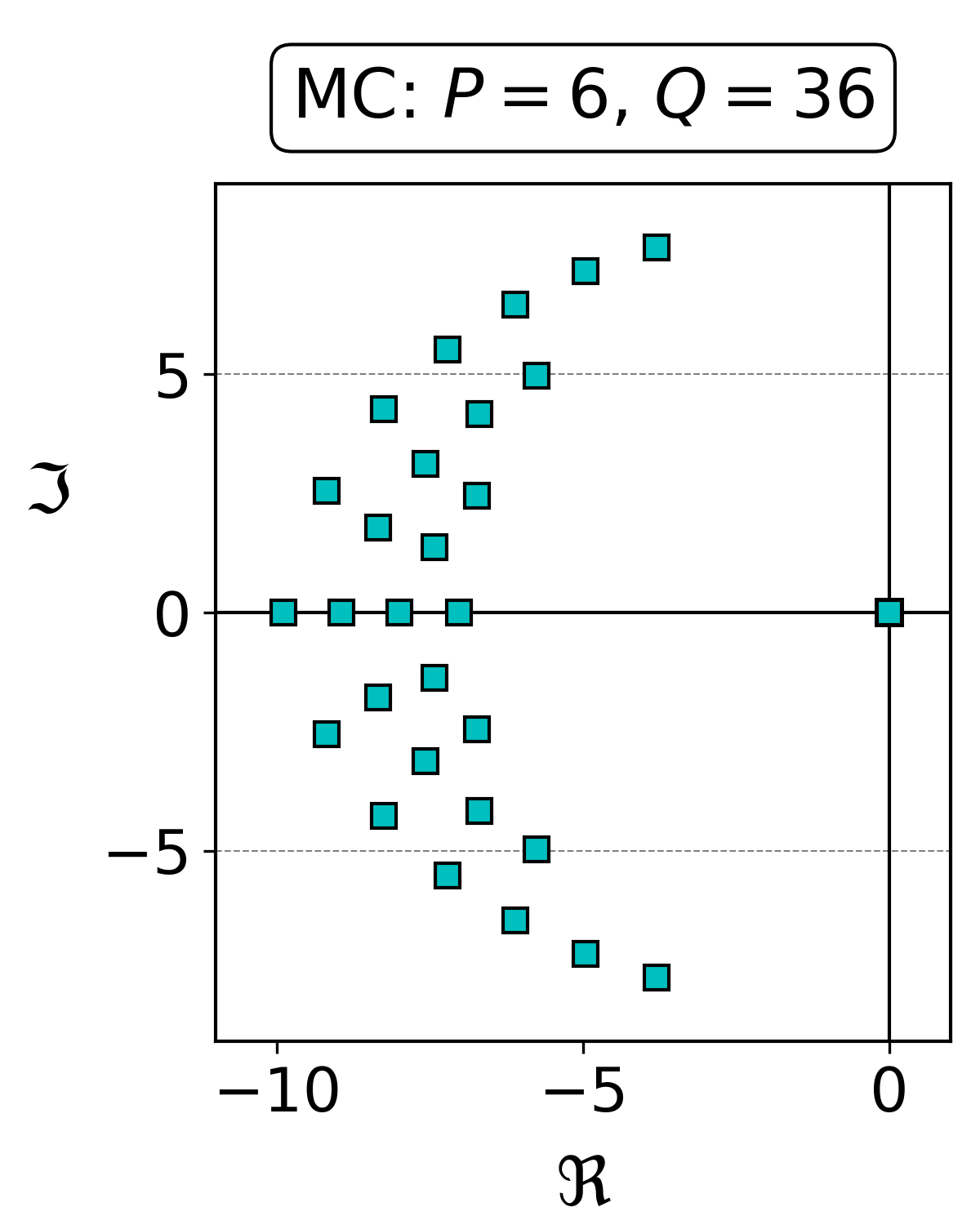}\\%
\includegraphics[width=0.32\textwidth]{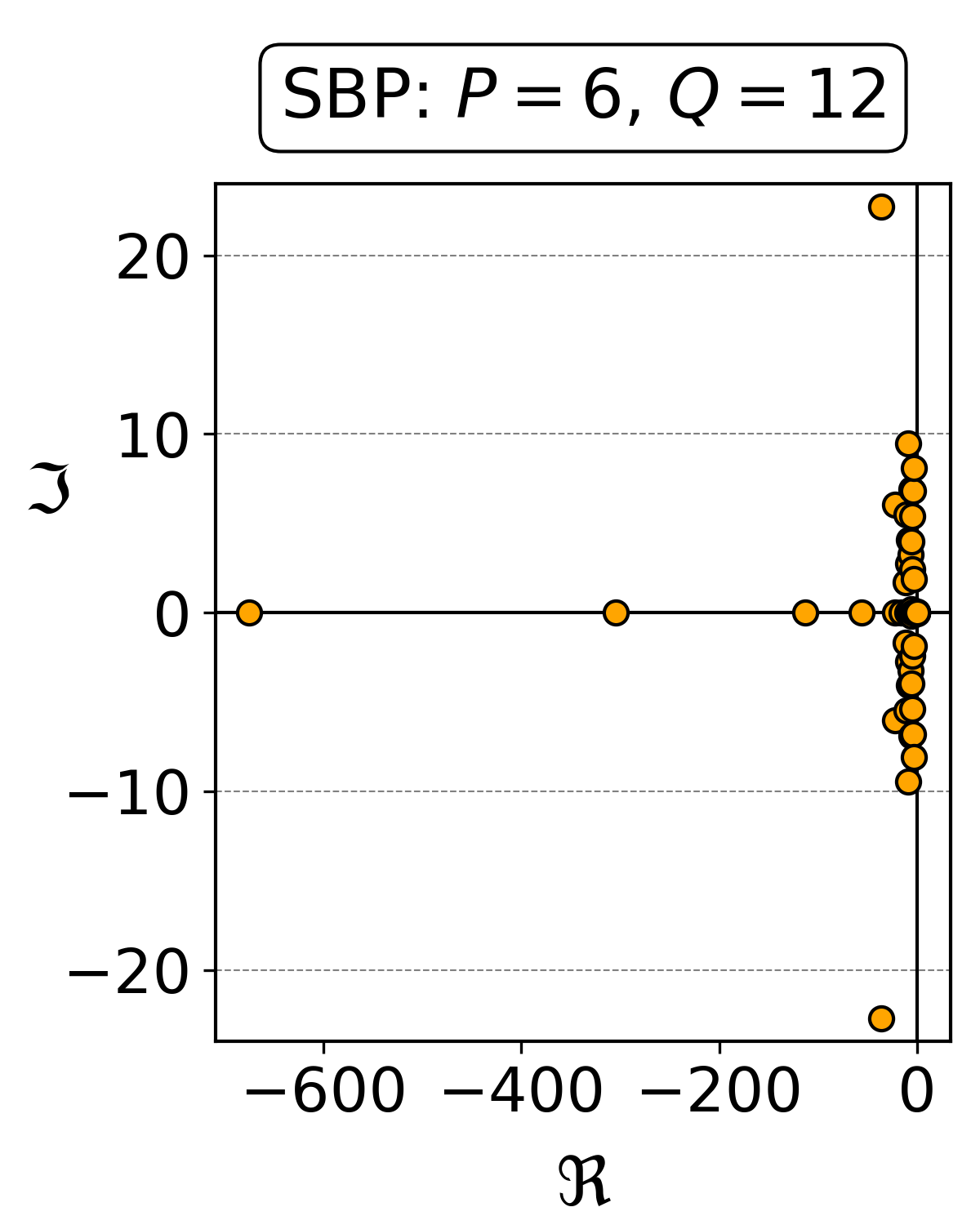}%
\includegraphics[width=0.32\textwidth]{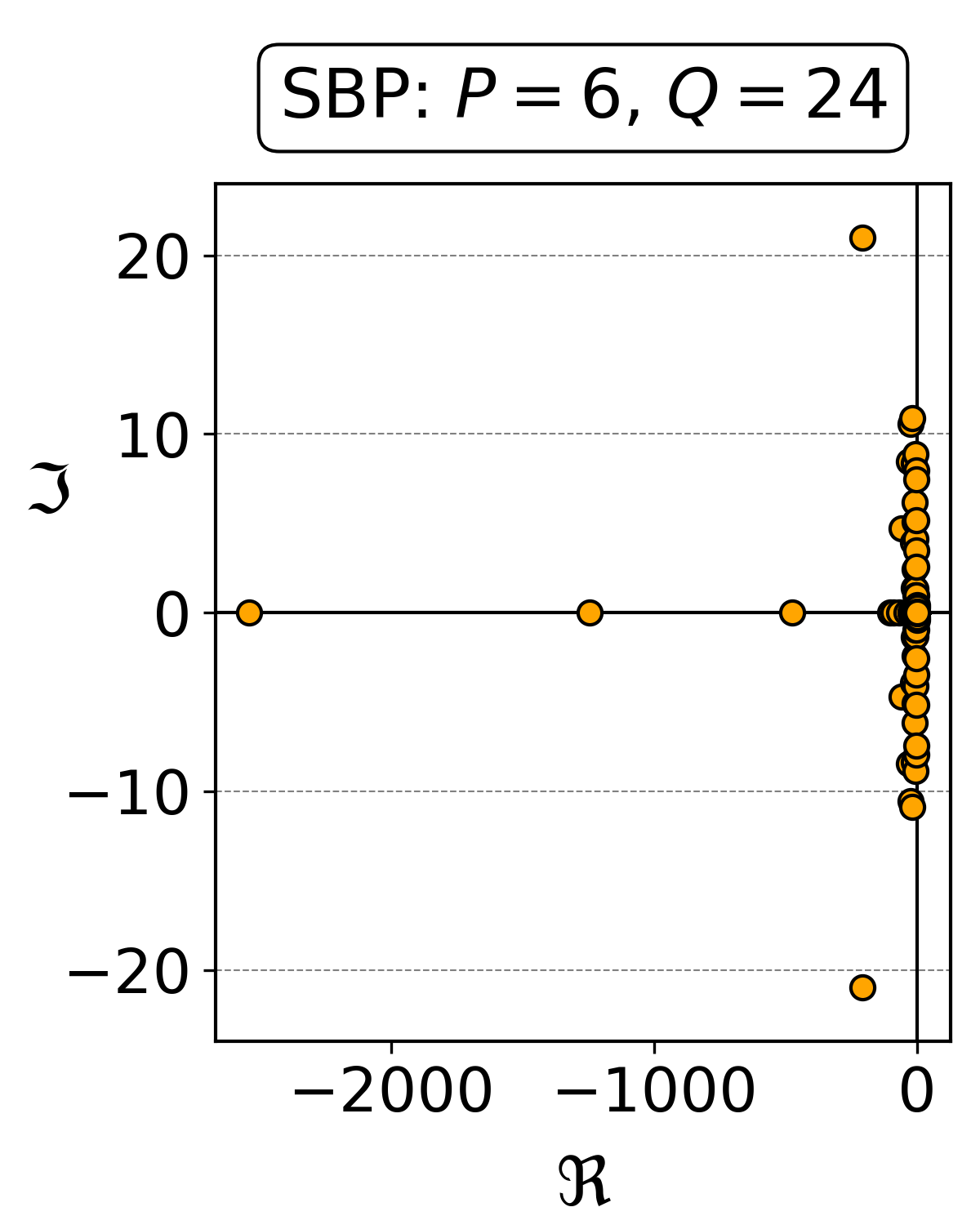}%
\includegraphics[width=0.32\textwidth]{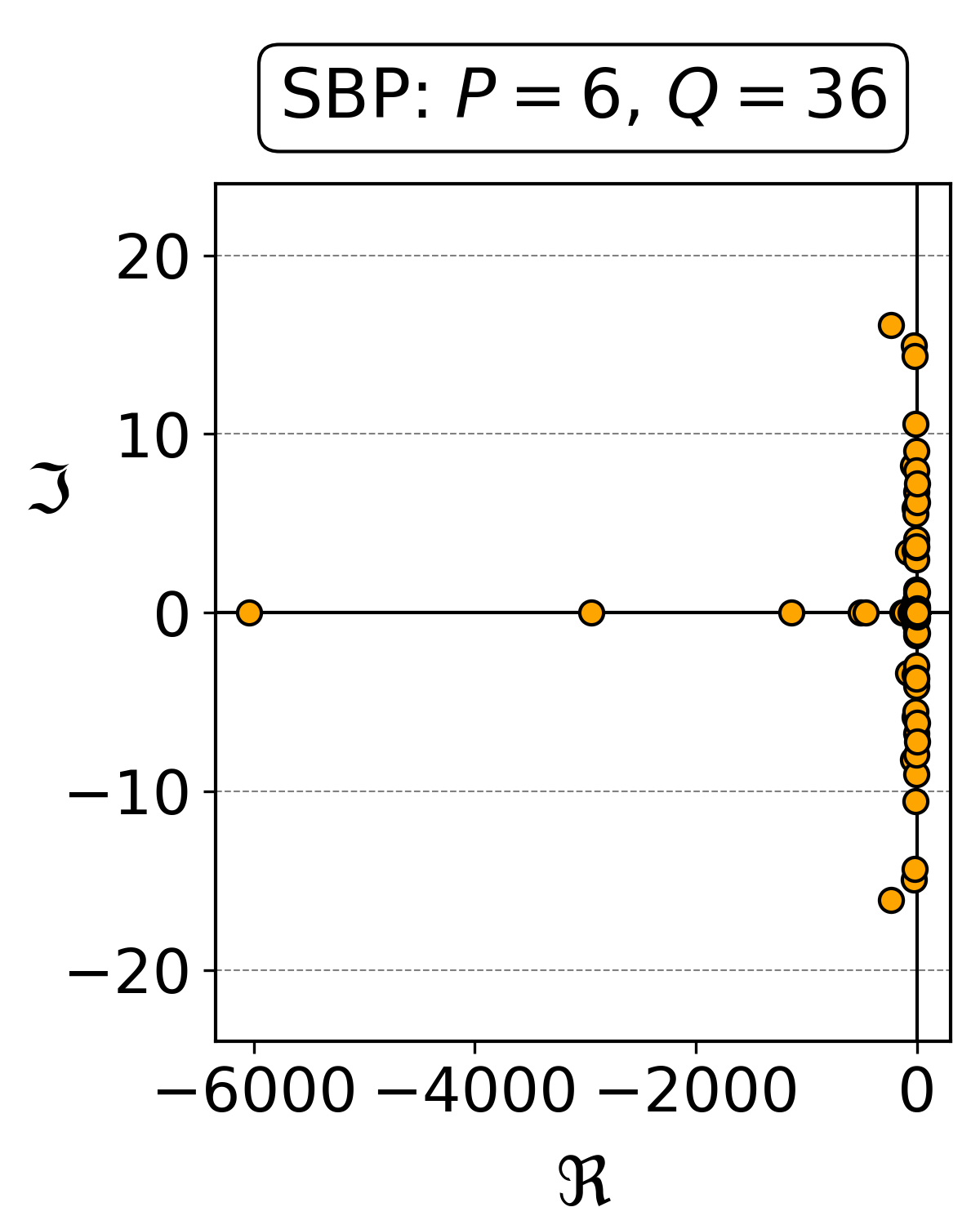}\\%
\caption{Eigenvalue distributions for the triangle-advection problem using $P=6$ MC (top) and SBP (bottom) operators with different quadrature degrees $Q$. \label{fig:spectrum_p6}}
\end{figure}

\paragraph{Equivalence and negative quadrature weights}

I conclude the constant-coefficient advection study with a short example of MC-DG equivalence when the quadrature has negative weights.  I use the same semi-discretization~\eqref{eq:advection_semidis} as before, but instead of constructing the MC operators using the collapsed tensor-product quadrature, I adopt quadrature 4c from \cite[Table 4]{Liu1998exact}.  This is an $N=10$ point quadrature that is exact for polynomials up to fourth degree, and I use it to construct an MC operator of degree $P=2$.

The quadrature nodes include the vertices of the triangle, and their corresponding weights are negative; the weights are $-1/30$ on the reference triangle.  Nevertheless, when I compare the MC and modal DG solutions at $t=2$ using the $L^2$ difference \eqref{eq:l2_diff}, I find that the two solutions match to machine precision; specifically $L^2\, \text{Diff.} = 3.9265 \times 10^{-16}$.

\begin{remark}
Conventional SBP energy-stability theory cannot predict the stability of the MC semi-discretization in this case, since the matrix $\mat{W}$ is indefinite.  Instead, stability of the MC solution is a consequence of its equivalence to the DG solution, which is well-known to be energy-stable for constant-coefficient advection; see, for example, \cite{Hesthaven2008nodal}.
\end{remark}

\subsection{Nonlinear initial boundary value problem}

\paragraph{IBVP and discretizations}

Next, consider the following two-dimensional version of Burgers' equation:
\begin{equation}\label{eq:burgers}
\begin{alignedat}{2}
    \frac{\partial \fnc{U}}{\partial t} &= -\frac{1}{2} \frac{\partial \fnc{U}^2}{\partial x_1},&\qquad &\forall\, \bm{x} \in \Omega, t \in [0,1], \\
    \fnc{U}(\bm{x},0) &= \bm{G}(\bm{x},0), &\qquad&\forall\, \bm{x} \in \Omega,
\end{alignedat}
\end{equation}
where $\Omega = [0,2\pi]^2$ is a square, periodic domain.  The initial condition is specified using the exact solution, $\fnc{G}(\bm{x},t)$, which is the root of the following nonlinear equation:
\begin{equation}
    \fnc{G}(\bm{x},t) - \sin\big(x_1 - t \fnc{G}(\bm{x},t)\big) = 0.
\end{equation}
It is easy to see that the Burgers' IBVP \eqref{eq:burgers} corresponds to the general IBVP~\eqref{eq:hyperbolic} by making the identifications $\fnc{F}_1(\fnc{U}) = \fnc{U}^2/2$, $\fnc{F}_2(\fnc{U}) = 0$, and $\fnc{S}(\fnc{U},t) = 0$.

I discretize the IBVP~\eqref{eq:burgers} on a multi-element mesh consisting of triangular elements.  To generate the mesh, I first create an $N_{1D}\times N_{1D}$ mesh of quadrilaterals and then split each quadrilateral into two triangles; thus, there are $K=2N_{1D}^2$ elements.  Figure~\ref{fig:burgers_mesh} shows an example of a mesh with $N_{1D} = 4$ using degree $P=2$ modal-collocation operators whose quadrature is $2P$ exact.

I investigate two MC-based discretizations for Burgers' equation.  The first, which I refer to as the Standard discretization, satisfies the requirements of Theorem~\ref{thm:equivalence}.  The second discretization is entropy conservative.

\begin{description}
\item[Standard:] The first discretization applies the MC semi-discretization~\eqref{eq:mc} on each element with zero source ($\bm{s} = \bm{0}$).  The $x_1$ and $x_2$ fluxes are given by $\bm{f}_{1} = \frac{1}{2} \bm{u} \circ \bm{u}$ and $\bm{f}_{2} = \bm{0}$, respectively, where $\circ$ denotes the entry-wise (Hadamard) product.  The numerical flux function in the Standard semi-discretization is 
\begin{equation*}
\big[ \bm{f}_{n}^{*}(\mat{R}_{\Gamma} \bm{u}, \bm{u}_\Gamma^{+}) \big]_i
= \frac{1}{2} \bigg( \frac{u_i + u_i^{+}}{2} \bigg)^2,
\qquad\forall\, i=1,2,\ldots,M,
\end{equation*}
where $u_i = \big[ \mat{R}_{\Gamma} \bm{u} \big]_i$ is the projection of the MC solution on the element to face node $i$, and $u_i^+$ is the projection of the solution to node $i$ from the appropriate neighboring element.

\item[Entropy Conservative:] The second discretization corresponds to a skew-symmetric split-formulation of Burgers' equation that is entropy conservative.  The discretization is given by Equation (8) in Crean et al.~\cite{Crean2018entropy}, and it uses the entropy conservative flux 
\begin{equation*}
    \fnc{F}_{x_1}^*(u_L, u_R) = \frac{u_L^2 + u_L u_R + u_R^2}{6}.
\end{equation*}
The discretization is entropy conservative, which means that the $L^2$ norm of the numerical solution is time invariant:
\begin{equation*}
\frac{d}{dt} \sum_{k=1}^{K} \bm{u}_k^T \mat{W}_k \bm{u}_k = 0,
\end{equation*}
where the subscript $k$ denotes the element index.
\end{description}
The results will demonstrate that the entropy-conservative discretization is not equivalent to modal DG, because it cannot be written in the form \eqref{eq:mc}; nevertheless, I will show how to recover equivalence, if it is desired.  

Figure~\ref{fig:burgers_solution} illustrates the initial and final solution using the Entropy-Conservative discretization with $N_1 = 16$ and elements with $P=2$ operators.  The numerical solution is evaluated over $y=\pi$ by projecting to the faces that coincide with the midline.  Since the numerical solution is multi-valued on the faces, the MC solution shown in Figure~\ref{fig:burgers_solution} corresponds to the simple average from adjacent elements.  Note that the solution using the Standard discretization is qualitatively similar.

\begin{figure}[tbp]
\begin{minipage}{.47\textwidth}
  \centering
    \includegraphics[width=\textwidth]{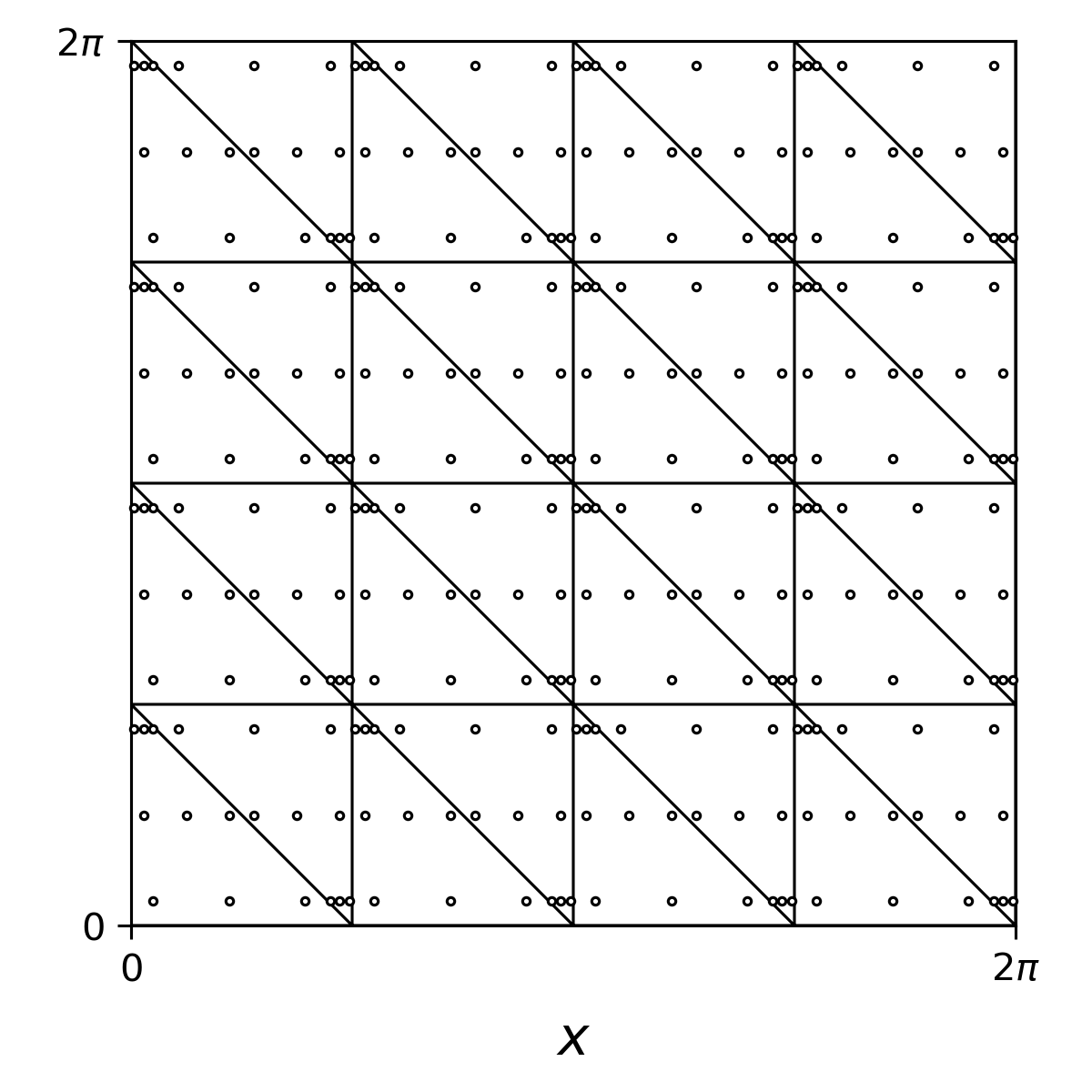}%
    \caption{Example mesh used for the Burgers' IBVP.\\ \label{fig:burgers_mesh}}
\end{minipage}\hfill%
\begin{minipage}{.47\textwidth}
    \centering
    \includegraphics[width=\textwidth]{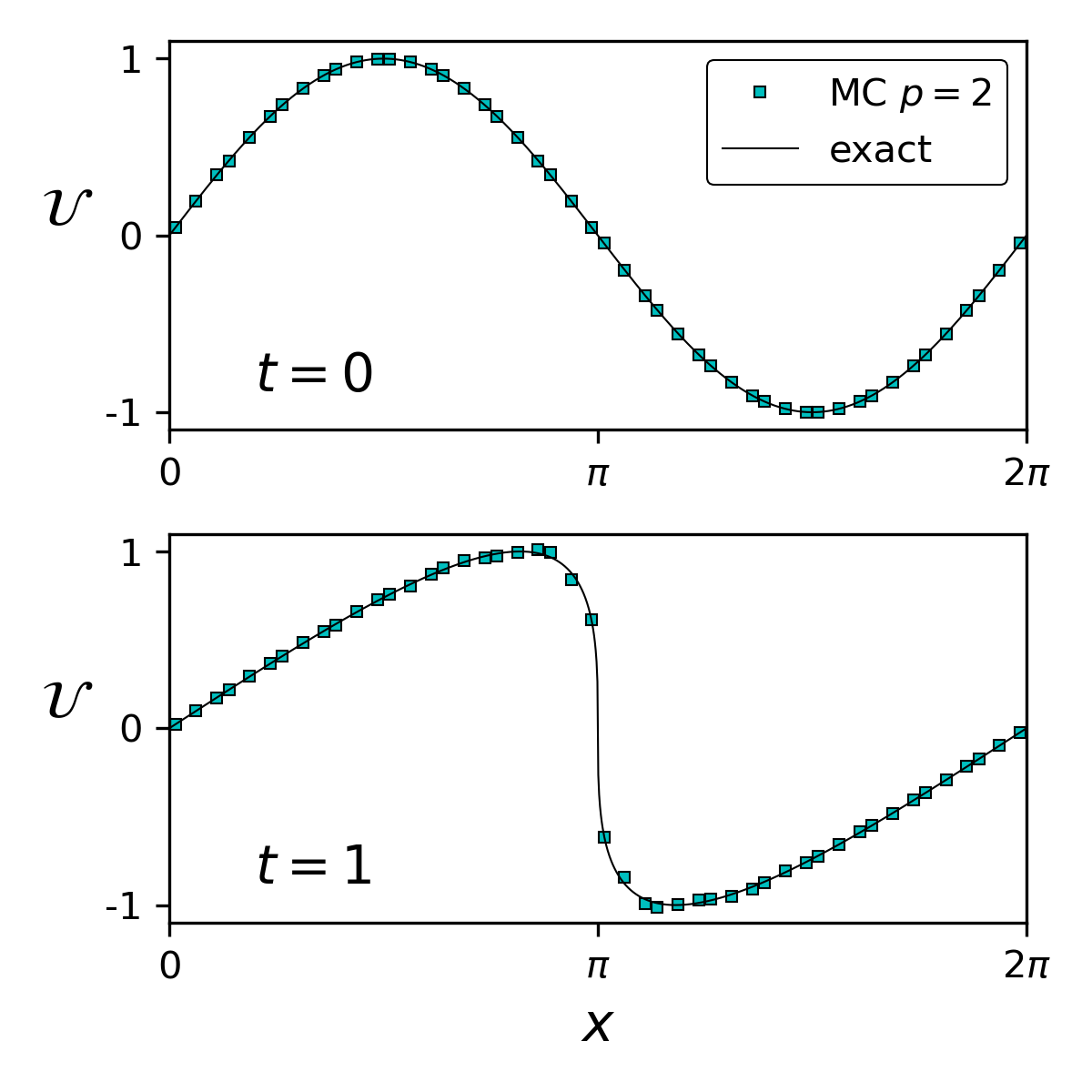}%
    \caption{Initial (top) and terminal (bottom) solutions for the Burgers' IBVP along $y=\pi$.\label{fig:burgers_solution}}
\end{minipage}
\end{figure}

\paragraph{Verification of accuracy}

Prior to investigating MC-DG equivalence for the Burgers' discretizations, I first verify the accuracy of the solutions produced by the two semi-discretizations.  I again use the low-storage method of Carpenter and Kennedy~\cite{Carpenter1994fourth} to advance the solution in time.  The time-marching scheme uses a fixed step size of
\begin{equation*}
    \Delta t = \frac{\text{CFL} \, h}{(P+1)^2 \fnc{G}_{\max}/2}
\end{equation*}
with a nominal CFL number of $1/2$, where $h = 2\pi/N_{1D}$ is the element size and $\fnc{G}_{\max} \equiv \max_{\bm{x}} \fnc{G}(\bm{x},0)$ is the maximum initial velocity.

Figure~\ref{fig:burgers_accuracy} plots the $L^2$ error for the Standard (Std) and Entropy-Conservative (EC) discretizations applied to the Burgers' equation.  The error is computed using the $L^2$ difference between the discrete and exact solutions:
\begin{equation}\label{eq:l2_error}
    L^2\; \text{Error} = \sqrt{\sum_{k=1}^K (\bm{u}_k - \bm{g}_k)^T \, \mat{W} \, \mat{J}_k \, (\bm{u}_k - \bm{g}_k)},
\end{equation}
where $\bm{g}_k$ denotes the exact solution at the nodes $\bm{x}_k$ of element $k$, and $\mat{J}_k$ is a diagonal matrix holding the determinant of the mapping Jacobian on element $k$ along its diagonal.  The figure includes mesh resolutions corresponding to $N_{1D} \in \{ 4, 8, 16, 32, 64\}$ and polynomial degrees $P \in \{1, 2, 3, 4\}$.  The results confirm that both discretizations produce high-order accurate solutions.

Table~\ref{tab:burgers_rates} lists the empirical rates of convergence for the schemes.  The rates are computed from the errors on the finest two mesh sizes, that is, $N_{1D} = 32$ and $N_{1D} = 64$.  Except for the odd degree Entropy-Conservative schemes, the errors converge at rates higher than $P+1$.  Note that both the Standard and Entropy-Conservative semi-discretizations do not include any numerical dissipation from numerical flux functions on the faces.


\begin{figure}[tbp]
\begin{minipage}{.47\textwidth}
  \centering
  \includegraphics[width=\textwidth]{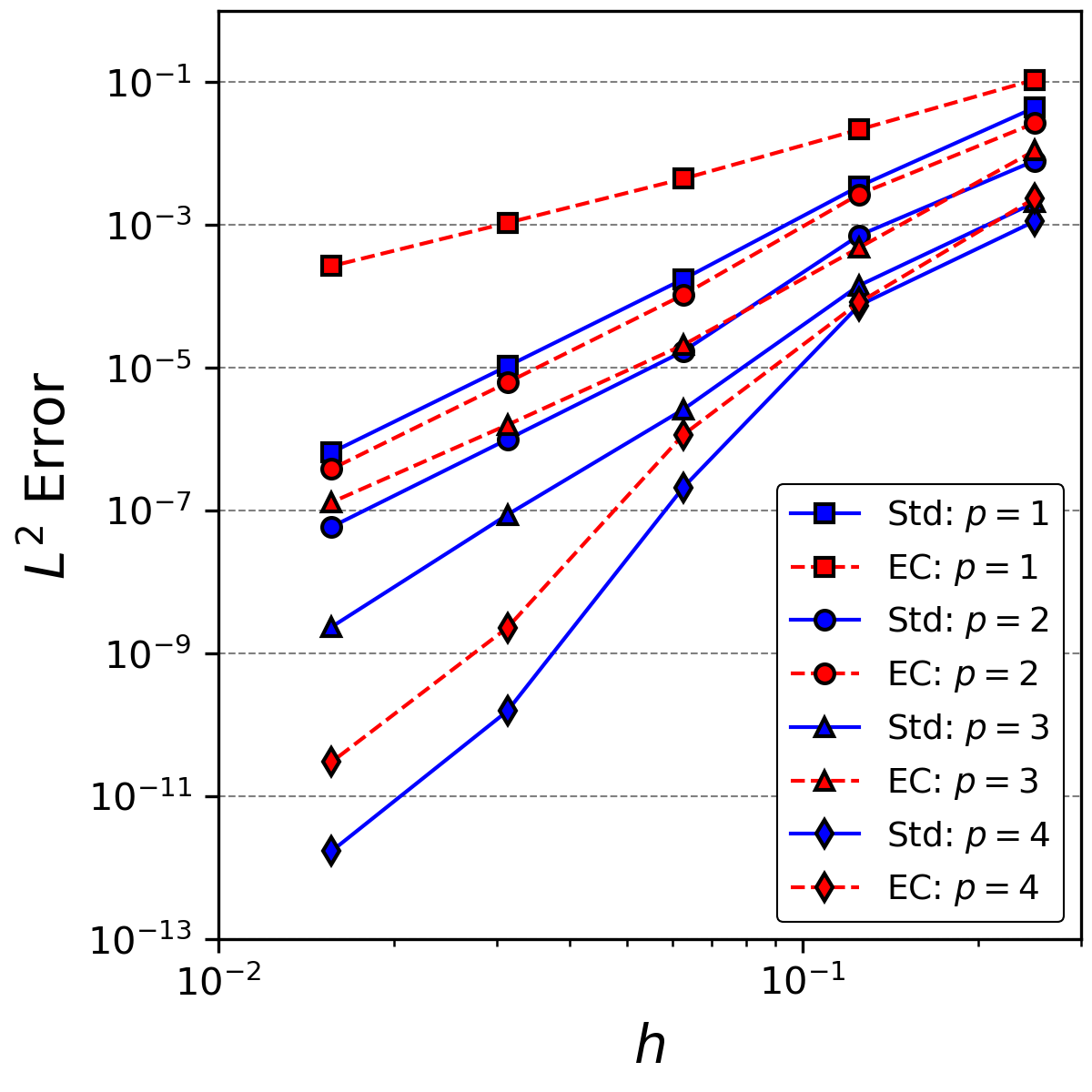}%
    \caption{$L^2$ errors as a function of element size $h = 1/N_{1D}$ for the Standard (Std) and Entropy Conservative (EC) discretizations of the Burgers equation.\label{fig:burgers_accuracy}}
\end{minipage}\hfill%
\begin{minipage}{.47\textwidth}
    \centering
    \includegraphics[width=\textwidth]{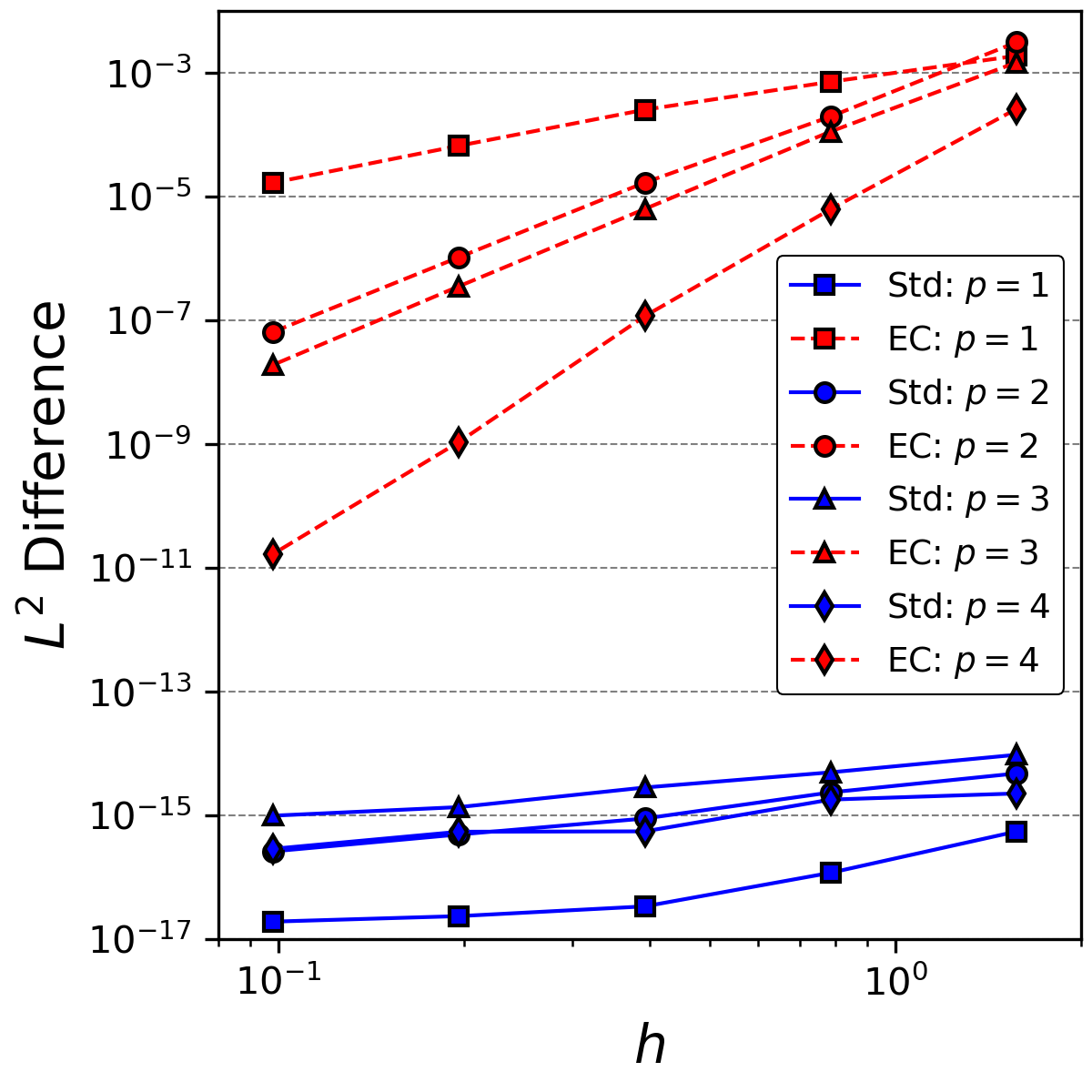}%
    \caption{$L^2$ differences as a function of element size $h = 1/N_{1D}$ between the two MC discretizations and the corresponding modal DG discretization.\label{fig:burgers_compare}}
\end{minipage}
\end{figure}

\begin{table}[tbp]
\centering 
\caption{Empirical rates of convergence for the Standard and Entropy Conservative discretizations of the Burgers equation. \label{tab:burgers_rates}}
\begin{tabular}{lllll}
& $P=1$ & $P=2$ & $P=3$ & $P=4$ \\\hline
\rule{0ex}{3ex}%
Standard & 4.011 & 4.069 & 5.221 & 6.519 \\[1ex] 
Entropy Conservative & 2.012 & 4.022 & 3.599 & 6.247 \\[1ex]\hline
\end{tabular}
\end{table}

\paragraph{Investigation of MC-DG equivalence}

In this section I compare the Standard and Entropy-Conservative MC discretizations with their modal DG counterparts.  As with the constant coefficient advection problem, the modal DG discretizations are obtained by substituting $\bm{u}_k(t) = \mat{V} \tilde{\bm{u}}_k(t)$ on each element $k$ in the MC semi-discretizations, and then left multiplying the element-based equations by $\mat{V}^T \mat{W}$.  I use affine transformations of the reference triangle, so the mapping Jacobian is constant over each element; consequently, the element mass matrices for the modal DG discretizations are simply scaled versions of the identity matrix.  The $L^2$ difference is evaluated using \eqref{eq:l2_error} with $\bm{g}_k$ replaced with the modal DG solution.

Figure~\ref{fig:burgers_compare} plots the $L^2$ difference between the Standard MC discretization of Burgers' equation and its corresponding modal DG discretization (blue curves), and the difference between the Entropy-Conservative discretization and its modal DG discretization (red curves). The results confirm that the Standard MC discretization is equivalent to the modal DG discretization, at least to machine precision.  By contrast, the Entropy-Conservative discretization differs from its modal DG discretization, although the difference converges to zero with mesh refinement.

The Entropy-Conservative discretization produces a residual that cannot be represented in the polynomial basis, which is why MC-DG equivalence is lost in this case.  Comparing the EC results in Figures~\ref{fig:burgers_accuracy} and \ref{fig:burgers_compare}, we see that difference between the discretizations is on the order of the error.  Consequently, the loss of MC-DG equivalence will not impact the results significantly.

That said, equivalence can be recovered by projecting the MC element residuals onto the polynomial basis, that is, by left multiplying the spatial part of the Entropy-Conservative discretization by $\mat{V} \mat{V}^T \mat{W}$.  To illustrate this, I ran the MC-DG comparison study with the projected Entropy-Conservative discretization on a fixed grid with $N_{1D} = 8$.  Table \ref{tab:burgers_compare} lists the $L^2$ differences between the MC and modal DG solutions at $t=1$ for $P \in \{1,2,3,4\}$ polynomial degrees.  To demonstrate that the equivalence is independent of quadrature accuracy, the table includes results for quadratures corresponding to $2P$, $4P$ and $6P$ polynomial exactness.  For all $P$ and quadrature combinations, we see that the projected Entropy-Conservative discretization has been rendered equivalent to modal DG.

\begin{table}[tbp]
\centering
\caption{$L^2$ difference between the projected Entropy-Conservative MC solution to the Burgers Equation and the corresponding modal DG solution.\label{tab:burgers_compare}}
\begin{tabular}{clll}
\rule{0ex}{5ex}
 & \multicolumn{3}{c}{\textbf{quadrature exactness}} \\ \cmidrule{2-4}
\textbf{degree $P$} & \multicolumn{1}{c}{$2P$} & \multicolumn{1}{c}{$4P$} & \multicolumn{1}{c}{$6P$} \\ \hline
\rule{0ex}{3ex}%
1 & 2.4023e-16 & 2.1587e-16 & 1.2093e-16 \\[1ex]
2 & 2.2058e-15 & 1.1910e-15 & 9.9250e-16  \\[1ex]
3 & 5.0207e-15 & 1.8121e-15 & 3.5166e-15 \\[1ex]
4 & 1.9155e-15 & 4.3904e-14 & 3.5165e-15 \\[1ex]\hline
\end{tabular}
\end{table}

\section{Summary and discussion}\label{sec:conclude}


I have highlighted some properties of the modal collocation (MC) operator introduced by Chan~\cite{Chan2018discretely}.  I showed that MC and DG semi-discretizations of a nonlinear hyperbolic IBVP produce equivalent solutions when they use the same underlying quadrature.  One consequence of this result is that the spectrum of MC semi-discretizations is effectively identical to the spectrum of modal DG for linear initial-boundary-value problems.



MC operators are nullspace inconsistent, in general.  I argued that this inconsistency is not an issue for unsteady problems, because the trivial modes remain zero.  On the other hand, for steady problems, an MC discretization should include local-projection stabilization to avoid singular Jacobians. 

I presented numerical studies to verify the equivalence between MC and DG discretizations in the case of the constant-coefficient advection equation and Burgers' equation.  Most of the results were based on a collapsed-coordinate tensor-product quadrature for the triangle; however, I also demonstrated MC-DG equivalence for the advection equation when the quadrature has negative weights.  I showed that an entropy-conservative MC discretization of Burgers' equation is not inherently equivalent to DG, because the residual is not spanned by the polynomial basis.  However, if desired, equivalence can be restored by projecting the residual onto the basis.
 

If MC and DG are effectively equivalent, which one should we use?  At this time I do not have a definitive answer.  Like many such questions regarding the use of one method versus another the answer may depend on the application.  For linear problems where the element mass and stiffness matrices can be precomputed, I suspect MC will provide no benefit and may be slower than a DG method.  On the other hand, it may be worth considering MC for nonlinear problems that are solved with explicit time-marching schemes or matrix-free implicit schemes.

Unlike the MC versus DG question, I am more confident answering the following: how should one construct multidimensional SBP operators?  Use modal collocation.

\section*{Declarations}

\paragraph{Funding}

This work was partially supported by the Office of Naval Research [grant number N00014-23-1-2698] with Dr. Mark Spector as the project manager.

\paragraph{Interests}

The author has no relevant financial or non-financial interests to disclose.

\paragraph{Code availability}

The code supporting this study is available upon reasonable request from the corresponding author.

\bibliography{references}

\end{document}